\documentclass[12pt, leqno]{article}
\usepackage{fullpage} 
\usepackage{amsmath}
\usepackage{amssymb}
\usepackage{placeins}
\usepackage{stmaryrd}
\usepackage{bbding}
\usepackage{enumerate}
\usepackage{graphicx}
\usepackage{xypic}
\xyoption{curve}
\usepackage{multicol}
\xyoption{color}
\usepackage{setspace}
\usepackage{boxedminipage}
\usepackage{nicefrac}
\usepackage{multicol}
\usepackage{amsthm}
  \usepackage[pdftex,
  bookmarks = false,
  pdfstartview = FitBH,
  linktocpage = true,
  pagebackref = true,
  pdfhighlight = /O,
  colorlinks=true,
  linkcolor=blue,
  citecolor=blue,
  filecolor = blue,
  urlcolor = blue,
  menucolor = blue,
]{hyperref}

\usepackage[T1]{fontenc}
\usepackage{textcomp} 

\theoremstyle{definition}

\makeatletter
\newtheorem*{rep@theorem}{\rep@title}
\newcommand{\newreptheorem}[2]{%
\newenvironment{rep#1}[1]{%
 \def\rep@title{#2 \ref{##1}}%
 \begin{rep@theorem}}%
 {\end{rep@theorem}}}
\makeatother

\newtheorem{thm}{Theorem}
\newreptheorem{thm}{Theorem}
\newtheorem{lem}[thm]{Lemma}
\newtheorem{prop}[thm]{Proposition}
\newtheorem{cor}[thm]{Corollary}
\newtheorem{defn}[thm]{Definition}
\newtheorem{Q}[thm]{Question}

\newtheorem{claim}[thm]{Claim}

\newcounter{parnum}

\newcounter{foo}

\newenvironment{myequation}
{\setcounter{equation}{\value{foo}}\begin{equation}}
{\setcounter{foo}{\value{equation}}\end{equation}\ignorespacesafterend}

\newenvironment{myeqnarray}
{\setcounter{equation}{\value{foo}}\begin{eqnarray}}
{\setcounter{foo}{\value{equation}}\end{eqnarray}\ignorespacesafterend}

\newenvironment{myenumerate}
{\begin{enumerate}\setcounter{enumi}{\value{foo}}}
{\setcounter{foo}{\value{enumi}}\end{enumerate}\ignorespacesafterend }

\numberwithin{equation}{section}
\numberwithin{enumi}{section}
\numberwithin{foo}{section}
\numberwithin{thm}{section}
\numberwithin{parnum}{section}

\title{Relative Categoricity and Abstraction Principles}
\author{Sean Walsh\footnote{Department of Logic and Philosophy of Science, 5100 Social Science Plaza, University of California, Irvine, Irvine, CA 92697-5100, U.S.A., swalsh108@gmail.com or walsh108@uci.edu} \; and Sean Ebels-Duggan\footnote{Department of Philosophy, Northwestern University, 1860 Campus Drive, Evanston, Illinois 60208-2214, U.S.A., s-ebelsduggan@northwestern.edu}}
\date{\today}

\begin{document}
\maketitle

\begin{abstract}
Many recent writers in the philosophy of mathematics have put great weight on the relative categoricity of the traditional axiomatizations of our foundational theories of arithmetic and set theory (\cite{Parsons1990a}, \cite{Parsons2008} \S{49}, \cite{McGee1997aa}, \cite{Lavine1999aa}, \cite{Vaananen2014aa}). Another great enterprise in contemporary philosophy of mathematics has been Wright's and Hale's project of founding mathematics on abstraction principles (\cite{Hale2001}, \cite{Cook2007aa}). In \cite{Walsh2012aa}, it was noted that one traditional abstraction principle, namely Hume's Principle, had a certain relative categoricity property, which here we term \emph{natural relative categoricity}. In this paper, we show that most other abstraction principles are \emph{not} naturally relatively categorical, so that there is in fact a large amount of incompatibility between these two recent trends in contemporary philosophy of mathematics. To better understand the precise demands of relative categoricity in the context of abstraction principles, we compare and contrast these constraints to (i) stability-like acceptability criteria on abstraction principles (cf. \cite{Cook2012aa}), (ii) the Tarski-Sher logicality requirements on abstraction principles studied by Antonelli \cite{Antonelli2010aa} and Fine~\cite{Fine2002}, and (iii) supervaluational ideas coming out of Hodes' work \cite{Hodes1984, Hodes1990aa, Hodes1991}.
\end{abstract}

\newpage

\tableofcontents

\newpage


\section{Introduction}\label{sec01}

Our topic in this paper is the compatibility of abstraction principles and relative categoricity. The most famous example of an abstraction principle is Hume's Principle, which asserts that the number~of~$X$'s is the same as the number~of~$Y$'s if and only if~$X$ and~$Y$ are equinumerous, where this just means that there is a bijection between the $X$'s and $Y$'s. In this, the ``number~of'' operator is understood to be a type-lowering operator which takes second-order entities and returns first-order entities. Much of the interest in Hume's Principle stems from the fact that it, in conjunction with certain axioms for second-order logic, 
recovers the now standard Peano axiomatization for  arithmetic (\cite{Wright1983} Chapter~4, cf. \cite{Walsh2012aa} \S{2.2} pp. 1688~ff). Wright vivified this in \cite{Wright1998ab} by asking us to consider the perspective of ``Hero,'' who using Hume's Principle  recovers all the laws of pure and applied arithmetic using Frege's definitions for zero, successor, and natural number.


The motivation for relative categoricity is likewise often conveyed by considering the perspective of agents. Parsons in his seminal essay \cite{Parsons1990a} and later book \cite[\S{49}]{Parsons2008}  asks us to consider interlocutors~$h$ and~$c$,  each having access to her own structure satisfying the axioms for arithmetic, but whose access to the other's number structure comes primarily from the other's literal utterances. Initially it might appear that the incompleteness of arithmetic could allow for intractable disagreement: perhaps the which structure agent~$h$ has in mind witnesses the arithmetized version of the consistency of some theory, while the structure agent~$c$ has in mind does not. 
However, Parsons notes that this can't happen so long as each interlocutor can in addition perform mathematical induction on concepts defined in terms of the other's natural number structure. If this is granted, then the map~$\Gamma$ defined by~$\Gamma(0_h)=0_c$ and~$\Gamma(s_h(n))=s_c(\Gamma(n))$ is an isomorphism, wherein~$0_i$ and~$s_i$ denote the zero and successor of agent~$i$'s number structure. Since isomorphisms preserve truth-values, the two agents will agree on all sentences of the pure arithmetical vocabulary in which zero and successor are taken as primitive.

So it's natural to ask whether there is similar agreement when it comes to abstraction principles.  To expand upon Wright's example, we might envision Hero as well as an interlocutor Claudio, and ask whether their agreement upon Hume's Principle leads to an agreement on other numerical truths. Frege famously noted in the \emph{Grundlagen} one potential source of disagreement:  Hero and Claudio might disagree about whether everything is a number, or whether this-or-that object is a number (\cite{Frege1884aa}, \cite{Frege1980} \S{56}, \S{66}). But it's still natural to ask what happens when we restrict attention to each 
agent's \emph{pure numbers}, that is, to the range of each's ``number~of'' operator. One important species of agreement has been stressed in the recent literature under the heading of the principle `Nq'. In terms of our scenario, Hero and Claudio would both think that the number~of~$Y$'s is equal to~$n$, as defined in terms of their respective ``number~of'' operators, if and only if~$\exists^{=n} \; x \; Yx$, wherein this is the exact numerical quantifier defined in the usual first-order way 
(cf. \cite{Hale1987aa} pp.~223-224, \cite{Wright1999} p.~18, \cite{Hale2001} p.~322, \cite{Cook2007aa} p.~32, \cite{Walsh2014aa}  p. 92).

It turns out that this agreement extends to the truth-values of \emph{all} pure 
numerical statements, where again we understand by ``pure'' the restriction to the  range of each  ``number~of'' operator. For, our interlocutors, just like Parsons', have a natural way of translating between their individual pure number discourse.  Designating Hero's ``number~of'' with ``$\#_h$'' and Claudio's with  ``$\#_c$,'' we see that 
whenever Claudio utters a statement about his pure numbers, we may replace each instance of~$\#_c$ by~$\#_h$ and obtain a truth about Hero's pure numbers. This is due to the fact that the map~$\Gamma(\#_h(X))=\#_c(X)$ is an isomorphism between Hero's and Claudio's ``pure number'' structures, so that we can again appeal to the fact that isomorphisms preserve the relevant truth-values. So not only will Hero and Claudio agree about all the truths of number theory, they will agree also, for instance, that the number~of evens is the same as the number~of natural numbers. 

This result about Hume's Principle was proven in the earlier paper \cite{Walsh2012aa} (Proposition 14 p. 1687).
However, this earlier work left open the question of whether this phenomena persists when one considers other abstraction principles. For, Wright and Hale \cite{Hale2001} have emphasized that Hume's Principle is just one abstraction principle amongst many. Other principles can be obtained by replacing the  equinumerosity relation with another equivalence relation on second-order entities and by introducing a new type-lowering operator for each such equivalence relation. So an \emph{abstraction principle} is a principle of the following form:
\begin{myequation}\label{eqn:AE}
A[E]: \hspace{10mm} \forall \; X, Y \; (\partial(X)=\partial(Y) \leftrightarrow E(X,Y))
\end{myequation}
In this, the quantifiers range over second-order entities and the operator~$\partial$ takes second-order entities and returns first-order entities. Moreover, the type-lowering operator $\partial$ is understood to depend on the equivalence relation~$E$. The case of the two agents is thus well-formalized by the following principle:
\begin{myequation}
A^2[E]: \hspace{5mm} \forall \; X, Y \; [(\partial_1(X)=\partial_1(Y) \leftrightarrow E(X,Y)) \wedge (\partial_2(X)=\partial_2(Y) \leftrightarrow E(X,Y))] \label{eqn:AE2}
\end{myequation} 
Let's call the objects in the range of the $\partial_i$-operator \emph{the abstracts} of interlocutor~$i$, and let's denote this by $\mathrm{rng}(\partial_i)$. Then we say that the abstraction principle~$A[E]$ is \emph{naturally relatively categorical} if it can be proved from this principle $A^2[E]$ (and the axioms governing the second-order logic) that the map $\Gamma(\partial_1(X))=\partial_2(X)$ is an isomorphism between the abstracts of the two interlocutors. For a more formal statement of natural relative categoricity, see Definition~\ref{eqn:defn:NRC}, which we provide after carefully setting up the particulars of the background second-order logic.

The present paper answers the question of what abstraction principles are naturally relatively categorical by presenting various equivalent characterizations in terms of sameness of cardinality and invariance under injections. These characterizations then allow us to  ascertain easily whether a given abstraction principle is naturally relatively categorical (see \S\ref{sec04} for examples). Let us build up to the statement of these equivalent characterizations by introducing some key definitions we deploy in this paper. First, let us say that the abstraction principle $A[E]$ is \emph{cardinality coarsening on abstracts} if the following is provable from the abstraction principle and the associated principles of the second-order logic: 
\begin{myequation}\label{eqn:defn:CC}
\forall \; X, Y\; ((Y\approx X \; \& \; X\subseteq \mathrm{rng}(\partial)) \rightarrow E(X,Y))]
\end{myequation}
In this, $Y \approx X$ is an abbreviation for the notion of equinumerosity  operative in Hume's Principle, while $\mathrm{rng}(\partial)$ is the collection of  objects in the range of the type-lowering operator~$\partial$; as above, we  sometimes refer to these as the \emph{abstracts}. Note that in~(\ref{eqn:defn:CC}), the concept $Y$ is not required to be subconcept of $\mathrm{rng}(\partial)$. Further, let us say that the abstraction principle $A[E]$ is \emph{injection invariant on abstracts} if the following is provable from the abstraction principle and the associated principles of the ambient second-order logic:
\begin{myequation}\label{eqn:defn:II}
\forall \; \mbox{ injection } \iota\hspace{-.5mm}:\hspace{-.5mm}V\hspace{-1mm}\rightarrow\hspace{-1mm} V \; \forall \; X\subseteq \mathrm{rng}(\partial) \;E(X,\overline{\iota}(X))
\end{myequation} 
wherein $V=\{x: x=x\}$ is an abbreviation for the concept of all objects and $\overline\iota (X)$ means $\{\iota(x): x \in X\}$. With this terminology in place, our primary characterization of natural relative categoricity can be stated as follows:
\begin{thm}\label{thm:ncr=iia=cca}
The following are equivalent:
\begin{myenumerate}
\item[] 1. The abstraction principle~$A[E]$ is naturally relatively categorical.\label{eqn:RC1}
\item[] 2. The abstraction principle~$A[E]$ is injection invariant on abstracts.\label{eqn:RC2}
\item[] 3. The abstraction principle~$A[E]$ is cardinality coarsening on abstracts.\label{eqn:RC3}
\end{myenumerate}
\end{thm}
\noindent This theorem is proven in \S{\ref{sec03}. Prior to establishing this theorem, in \S\ref{sec02} we  define the particulars of the second-order logic which we're employing-- in short, we're assuming full impredicative comprehension and strong forms of the axiom of choice. Then in \S\ref{sec02.5} we present the official definition of natural relative categoricity in Definition \ref{eqn:defn:NRC}. 

As mentioned above, Frege observed that abstraction principles like Hume's Principle don't determine whether or not everything is an abstract. 
The idea behind the natural relative categoricity of Hume's Principle is that this is the only kind of statement-- expressible in the language of Hume's Principle-- whose truth-value is left undetermined by Hume's Principle. So in natural relative categoricity we restrict attention down to the abstracts. A complementary idea is to restrict attention to the case where the abstraction operator is assumed to be a surjective map from concepts to objects. To this end, let us call an abstraction principle $A[E]$ \emph{surjectively relatively categorical} 
if it can be proved (in the background logic) from   $A^2[E]$, and the claim that $\partial_1$ and $\partial_2$ are surjective, that the map $\Gamma(\partial_1(X))=\partial_2(X)$ is an isomorphism between the abstracts of the two interlocutors. 
In analogy with our first main Theorem~\ref{thm:ncr=iia=cca}, our second main theorem establishes the following equivalent characterization of surjective relatively categorical abstraction principles:
\begin{thm}\label{thm:coverthm}
The following are equivalent:
\begin{myenumerate}
\item[] 1. The abstraction principle $A[E]$ is surjectively relatively categorical.
\item[] 2. The abstraction principle $A[E]$ is permutation invariant, 
under the assumption that the abstraction operator is surjective.
\item[] 3. The abstraction principle $A[E]$ is bicardinality coarsening, under the assumption that the abstraction operator is surjective.
\end{myenumerate}
\end{thm}
\noindent In analogy to injection invariance on abstracts~(\ref{eqn:defn:II}), we say that $A[E]$ is \emph{permutation invariant} if
\begin{myequation}\label{eqn:defn:piI}
\forall \; \mbox{ bijection } \pi\hspace{-.5mm}:\hspace{-.5mm}V\hspace{-1mm}\rightarrow\hspace{-1mm} V \;\; \forall \; X \;E(X,\overline{\pi}(X))
\end{myequation}
Further, in analogue to cardinality coarsening on abstracts~(\ref{eqn:defn:CC}), we say that $A[E]$ is \emph{bicardinality coarsening} if
\begin{myequation}\label{eqn:defn:biC}
\forall \; X, Y\; (Y\approx X \; \& \; V\hspace{-.5mm}\setminus\hspace{-.5mm}X\approx V\hspace{-.5mm}\setminus\hspace{-.5mm}Y) \rightarrow E(X,Y)
\end{myequation}
In this, $V\setminus X$ denotes the concept of all objects which are not in $X$. It's worth emphasizing that in the statement of Theorem~\ref{thm:coverthm}, the 
last two conditions occur under the hypothesis that the abstraction operator~$\partial$ is surjective: for every object~$b$ there is a concept~$F$ such that $\partial(F)=b$. Hence there is no restriction to abstracts in the formulations of permutation invariance~(\ref{eqn:defn:piI}) and bicardinality coarsening~(\ref{eqn:defn:biC}). 
As with our earlier theorem, Theorem~\ref{thm:coverthm} is proven in  \S{\ref{sec03};  the formal definition of surjective relative categoricity is given in Definition~\ref{eqn:defn:NRC:covered}, subsequent to our treatment of the background second-order logic in the next section.

The aim of our Theorem~\ref{thm:ncr=iia=cca} and Theorem~\ref{thm:coverthm} is to answer the question of what abstraction principles are relatively categorical in the specified senses, and in the coarsest of terms these theorems indicate that such abstraction principles look a lot like Hume's Principle, so that relative categoricity amongst abstraction principles is the exception rather than the rule. This obviously doesn't directly imply that there's a problem with appeals to either abstraction principles or relative categoricity considerations. But it indicates that a choice must be made: the advocate of relative categoricity arguments will find her preferred route to determinacy of truth value blocked in the case of most abstraction principles, and the advocate of abstraction principles might be pressed to find some other means by which to secure determinacy of truth value. That said, it's obviously non-trivial to spell out precisely what philosophical concern or question is intended to be assuaged by securing determinacy of truth-value (cf. \cite{Button2014ab}), and there are similarly problems with making out the case that abstraction principles can secure knowledge of our foundational theories of arithmetic and set theory (cf. \cite{Walsh2014aa}). This is not the place to adjudicate these larger philosophical issues. Rather, the aim of this paper is limited to showing that the tools which one segment of the philosophy of mathematics community have been using are largely incompatible with the tools employed by another part of the community. And this despite the fact that both relative categoricity and abstraction principles can be seen as latter-day descendants of the idea that the subject-matter of mathematics is given by implicit definitions of its fundamental concepts (cf. \cite{Hale2000}, \cite{Shapiro2005aa} p. 13, pp. 168-169, \cite{Shapiro1991} p. 190, \cite{Shapiro2000ac} pp. 132 ff).

The present paper is organized as follows. In \S\ref{sec02} we set out the particulars of the background second-order logic which we're employing. Then in \S\ref{sec02.5} we make a preliminary study of the map $\Gamma(\partial_1(X))=\partial_2(X)$, which we there call \emph{the natural bijection}. In the subsequent section \S\ref{sec03} we prove Theorem~\ref{thm:ncr=iia=cca} and Theorem~\ref{thm:coverthm}, as well as note some related results on the underlying equivalence relations. In \S\ref{sec04}, we use these theorems to  determine quickly whether some well-known abstraction principles are 
relatively categorical in the senses  we have adumbrated. In \S\S\ref{sec05stable}-\ref{sec07} we contrast the nature of the requirement of relative categoricity to other constraints on abstraction principles related to invariance and determinacy of truth-value studied by authors such as Cook, Antonelli, Fine, and Hodes. In particular, in \S\ref{sec05stable} we indicate where relatively categorical abstraction principles fit into the stability hierarchy that has arisen in response to the Bad Company problem. In \S\ref{sec05:aldo}, we distinguish our notion of permutation invariance~(\ref{eqn:defn:piI}) from notions studied by Antonelli and Fine and related to the Tarski-Sher thesis on logicality. Finally, in \S\ref{sec05:hodes} we note that the determinacy of truth value ideas coming out of our notions of relative categoricity might be orthogonal to the determinacy ideas coming out of Hodes' supervaluationism (cf. Question~\ref{eqn:Q2}).

\section{Background Second-Order Logic}\label{sec02}

We work in a background second-order signature which contains a sort for objects as well as a sort for~$n$-ary relations for each~$n\geq 1$. Objects are written with lower-case roman letters~$a,b,c,d,x,y,z,\ldots$.  The unary relations are called \emph{concepts} and written with upper-case roman letters~$A,B,C,X,Y,Z$, while~$n$-ary relations for~$n>1$ are written with upper-case roman letters~$R,S$. The predication relation is written~$Xa$ or~$a\in X$ for objects~$a$ and concepts~$X$; and it is typically written~$R(a_1, \ldots, a_n)$ for~$n$-ary relations. For the sake of definiteness, let us then stipulate:
\begin{defn}\label{defn:L000}
The background second-order signature $L_0$ is the many-sorted signature which consists merely of (i) sorts for objects and $n$-ary relations for each $n\geq 1$, and (ii) for each $n\geq 1$, the $(n+1)$-ary predication relations $R(x_1, \ldots, x_n)$ wherein $R$ is an $n$-ary relation and $x_1, \ldots, x_n$ are objects.
\end{defn}
\noindent The models of $L_0$ thus have the following form:
\begin{myequation}\label{eqn:defn:models}
\mathcal{M}=(M, S_1[M], S_2[M], \ldots)
\end{myequation}
wherein~$M$ is a non-empty set and~$S_1[M]\subseteq P(M^n)$, and wherein the predication relations are interpreted with the membership relation from the ambient set-theoretic metatheory. Often in what follows we will be discussing isomorphisms of $L_0$-structures and related expansions. In this connection, it's useful to explicitly note that $L_0$ itself does not contain any constant or function symbols and does not contain any relation symbols besides the predication relation symbols.

Suppose that~$L$ is an expansion of~$L_0$.  An $L$-structure~$\mathcal{M}$ whose $L_0$-reduct is written as in~(\ref{eqn:defn:models}) is called \emph{standard} if~$S_n[M]=P(M^n)$; we do \emph{not} assume here that all structures are standard. The \emph{full comprehension schema} for concepts in~$L$ is the collection of all the following axioms:
\begin{myequation}\label{eqn:fullcomp}
\exists \; X \; \forall \; x \; (Xx \leftrightarrow \Phi(x))
\end{myequation}
wherein~$\Phi(x)$ is an~$L$-formula and~$X$ does not appear free in~$\Phi(x)$, but where~$\Phi(x)$ may contain other free variables which are reserved for parameters. There are similar comprehension schemas for the~$n$-ary relations. 

In this paper, it is assumed that all theories contain the full comprehension schema in their signature for~$n$-ary relations for all~$n\geq 1$. Our primary motivation for working with the full comprehension schema in this paper is that it is presupposed by the statement of natural relative categoricity, as we will make clear in the formal presentation of this notion below. The comprehension schema allows us to use usual boolean connectives~$A\cap B$,~$A\cup B$,~$A\setminus B$,~$A\subseteq B$ with their usual meanings on both concepts and~$n$-ary relations for~$n>1$. Likewise, we use $A\times B$ to denote the binary concept consisting of pairs $(a,b)$ where $a$ is from $A$ and $b$ is from $B$. Sometimes in what follows we use the disjoint union notation~$A=B\sqcup C$, which of course just means that~$A=B\cup C$ and~$B\cap C=\emptyset$. In this paper,~$\emptyset$ and~$V$ are reserved for the concept of no objects and the concept of all objects, respectively:
\begin{myequation}\label{eqn:defnV}
\emptyset = \{x: x\neq x\}, \hspace{10mm} V = \{x:x=x\}
\end{myequation}

In what follows, we often employ various abbreviations for formulas in second-order logic. In particular, we use bracket notation~$\{x: \Phi(x)\}$ as short-hand for the unique concept~$X$ determined by~$\Phi(x)$ from the comprehension axiom in equation~(\ref{eqn:fullcomp}), and likewise we write~$\{(x_1, \ldots, x_n): \Phi(x_1, \ldots, x_n)\}$ in the case of~$n$-ary relations. 
Functions are identified with their graphs, so that~$G:A\rightarrow B$ is an abbreviation for the claim that the binary relation~$G$ is such that for all~$a$ from~$A$ there is unique~$b$ from~$B$ with~$G(a,b)$. Likewise we have the following abbreviations for cardinality-related notions:
\begin{myenumerate}
\item ~$X\approx Y$ or~$\left|X\right|=\left|Y\right|$ is an abbreviation for there being a bijection~$F:X\rightarrow Y$.\label{eqn:cardinalityabb:1}
\item $\left|X\right|\leq \left|Y\right|$ is an abbreviation for there being an injection~$F:X\rightarrow Y$. \label{eqn:cardinalityabb:2}
\item $\left|X\right|<\omega$ is an abbreviation for the claim that~$X$ is Dedekind-finite, i.e. any injection~$F:X\rightarrow X$ is also a surjection; and we abbreviate~$\left|X\right|\geq \omega$ for its negation.\label{eqn:cardinalityabb:3}
\end{myenumerate}
Another abbreviation which we shall employ repeatedly in what follows pertains to images of concepts under maps on objects. Suppose that $f:X\rightarrow Y$ is a map and suppose $X_0\subseteq X$. Then we define the image of $X_0$ under $f$ as $\overline{f}(X_0)=\{f(x): x\in X_0\}$, which is a concept by full comprehension when $f\in S_2[M]$. Typically we apply this in the case where $f:V\rightarrow V$, so that $\overline{f}(X)$ is defined for any concept $X$. 

The other definitions that we need in order to state our results are two forms of the axiom of choice. The first form is designated as ${\tt AC}$, and it is the following schema:
\begin{myequation}\label{eqn:AC}
[\forall \; \overline{x} \; \exists \; R^{\prime} \; \varphi(R^{\prime},\overline{x})]\rightarrow \exists \; R \; [\forall \; \overline{x} \; \varphi(R[\overline{x}],\overline{x})]
\end{myequation}
wherein $R[\overline{x}]=\{\overline{y}: R\overline{x}\overline{y}\}$, which exists by full comprehension. Intuitively this says that if for every $n$-tuple of objects $\overline{x}$ there is an $m$-ary concept $R^{\prime}$ witnessing the condition $\varphi(R^{\prime},\overline{x})$, then there is an $(n\mbox{+}m)$-ary concept $R$ such that for all $n$-tuples $\overline{x}$ the $m$-ary concept $R[\overline{x}]$ is a witness. More intuitively still, this version of the axiom of choice says that if for each object there is a concept satisfying a certain condition, then there is a uniform way to select these concepts. This version of the axiom of choice was used frequently in the earlier papers (\cite{Walsh2012aa} Definition 5 p. 1683, \cite{Walsh2014ac}) since in the setting of limited comprehension it is a natural component of a sufficient condition for the so-called $\Delta^1_1$-comprehension schema. 

The other form of the axiom of choice that we employ is a form of global choice. Suppose that~$T$ is a theory in one of our signatures. Then we let~$T\mbox{+}{\tt GC}$ be the expansion of~$T$ by a new binary relation symbol~$<$ on objects in the signature, with axioms saying that~$<$ is a linear order of the first-order objects, and we additionally have a schema in the expanded signature saying that any instantiated formula~$\varphi(x)$ in the expanded signature, perhaps containing parameters, that holds of some first-order object~$x$ will hold of a~$<$-least element: 
\begin{myenumerate}
\item~$[\exists \; x \; \varphi(x)]\rightarrow [\exists \; x \; \varphi(x) \; \& \; \forall \; y<x \; \neg \varphi(y)]$ \label{eqn:gcschema}
\end{myenumerate}
Since all our theories~$T$ contain full comprehension~(\ref{eqn:fullcomp}), we have that the graph of~$<$ forms a binary concept in ~$T\mbox{+}{\tt GC}$. Of course the postulated binary relation~$<$ does not necessarily have anything to do with the the usual ``less than'' relation on the natural numbers. This form of global choice was also defined and employed in the paper \cite{Walsh2014ac} where it was likewise designated as ${\tt GC}$. So in contrast to equation~(\ref{eqn:defn:models}), models of our global choice principle ${\tt GC}$ have the form:
\begin{myequation}\label{eqn:defn:models2}
\mathcal{M}=(M, S_1[M], S_2[M], \ldots, <)
\end{myequation}
where $<$ is a linear order on $M$ such that any non-empty $\mathcal{M}$-definable subset has a least element.

It's worth stressing that $L_0$ does \emph{not} include the global well-order (cf. Definition~\ref{defn:L000}). For, the equivalence relations that we will consider will all be $L_0$-formulas, and in \S\ref{sec05:aldo}, we will note that this implies that the equivalence relations are logical in the sense of Tarski-Sher, and this would be not be true if they included the global well-order. However, as far as theories and structures go, in what follows, it is assumed that all theories and structures contain ${\tt AC}$ and ${\tt GC}$. This of course may be taken to hold for  standard structures by recourse to the axiom of choice in the metatheory. Since we are additionally assuming full comprehension, our models in effect look and act a lot like standard models. However, the advantage of working with arbitrary models of these axioms-- instead of restricting ourselves to the standard models-- is that we have the the benefits of the completeness theorem (cf. \cite{Enderton2001}~Chapter~4, \cite{Manzano1996}~Chapter~VII.2). So even though we are working model-theoretically, everything can in principle be turned into a concrete deduction by recourse to this theorem. Because  the semantics for second-order logic are a contentious affair 
(cf. \cite{Linnebo2011aa} for overview), and because relative categoricity arguments have been traditionally motivated by concerns with the standard semantics for second-order logic (\cite{Parsons2008} p. 270, \cite{McGee1997aa} pp. 45-47, \cite{Lavine1999aa} p. 5), it is useful to adopt a framework in which  our results do not depend  on the choice of semantics for second-order logic.

Our use of ${\tt AC}$~(\ref{eqn:AC}) is rather limited in this paper: we appeal to it to establish the implication recorded in Figure~\ref{diagram1} in \S\ref{sec03} which allows us to go from conditions on an equivalence relation to conditions on the associated abstraction principle; and we appeal to it a final time to treat notions of finiteness in our discussion of the examples in \S\ref{sec04}. As for global choice ${\tt GC}$~(\ref{eqn:gcschema}), the reason why we assume it in this paper is that it permits us to transfer the usual properties of cardinal arithmetic to our deductive setting. In particular, in what follows we make use of the following three properties:
\begin{myenumerate}
\item \emph{Cardinal Comparability} :~$\forall \; X,Y \; (\left|X\right|\leq \left|Y\right| \vee \left|Y\right|\leq \left|X\right|)$\label{eqn:CC}
\item \emph{Infinite Sums are Maxs}:~$\forall \; X \; \left|X\right|\geq \omega \rightarrow [\forall \; Y, Z \; (X=Y\sqcup Z \rightarrow \left|X\right|=\max\{\left|Y\right|, \left|Z\right|\})]$\label{eqn:InfSumMax}
\item \emph{Infinite Products are Maxs}:~$\forall \; Y,Z \; (\left|Y\right|\geq \omega \vee \left|Z\right|\geq \omega)\rightarrow \left|Y\times Z \right| = \max\{\left|Y\right|, \left|Z\right|\}$\label{eqn:InfProdMax}
\end{myenumerate}
Of course, by Cardinality Comparability~(\ref{eqn:CC}), the maximum expressions make good sense. For instance, the clause~$\left|X\right|=\max\{\left|Y\right|, \left|Z\right|\}$ is just an abbreviation for the following conjunction of conditionals, and Cardinality Comparability~(\ref{eqn:CC}) implies that one of the two antecedents is satisfied:
\begin{myequation}
(\left|Y\right|\geq \left|Z\right| \rightarrow \left|X\right|= \left|Y\right|) \wedge (\left|Z\right|\geq \left|Y\right| \rightarrow \left|X\right|= \left|Z\right|)
\end{myequation}
Note that Cardinality Comparatibility~(\ref{eqn:CC}) follows from global choice ${\tt GC}$: for, the global well-order restricted to any two concepts yields two well-orders, and we can then use the traditional proof that well-orders are either order-isomorphic or one is isomorphic to an initial segment of the other (cf. \cite{Hrbacek1999aa} p. 105, \cite{Kunen1980} p. 15); and this result trivially implies that each is comparable to the other in terms of cardinality as well. Similarly, we can transfer the usual proof of Infinite Sums are Maxs and Infinite Products are Maxs (\ref{eqn:InfSumMax})-(\ref{eqn:InfProdMax}) as expressed in the language of set theory to our framework for second-order logic, since this proof proceeds by transfinite induction, which we can emulate with our global well-order (cf. \cite{Hrbacek1999aa} pp. 134 ff, \cite{Kunen1980} p. 29, \cite{Kunen2011aa} p. 73).

Another use of the global well-order $(V,<)$ that we employ is in giving an equivalent characterization of Dedekind-finiteness~(\ref{eqn:cardinalityabb:3}). First, let's introduce the following notation for the initial segments: 
$I_a=\{b: b<a\}$ and $\overline{I}_a=\{b: b\leq a\}$. Further, since~$(V,<)$ is a non-empty well-order, it has a least element, which we designate as zero or~$0$. There's also a natural partial successor function~$s$ defined as follows:
\begin{myequation}\label{eqn:defnsucconord}
s(a)=\min_{<}(V\setminus \overline{I}_a)=\min_<\{b: b>a\}
\end{myequation}
This function~$s$ might be partial because there might be a greatest element in well-order\; $(V,<)$. Finally, let's say that a \emph{limit point} in~$(V,<)$ is a point~$a>0$ such that~$b<a$ implies~$s(b)<a$. As in the theory of ordinals, the well-order~$(V,<)$ splits into zero, successors, and limits. Finally, let's say that $a$ is \emph{finite} if $a$ is strictly below the all the limit points, and let's say that $X$ is \emph{finite} if $X$ is bijective with $I_a$ or $\overline{I}_a$ for some finite $a$. Then one can show using induction that:
\begin{prop}\label{eqn:equivchardedfin}
(i) $X$ is finite if and only if $\left|X\right|<\omega$. (ii) If there is a least limit, and the least limit point is designated as $\omega$, then $X$ is infinite iff there is an injection $\iota:I_{\omega}\rightarrow X$, which happens iff $\left|X\right|\geq \omega$.
\end{prop}
\noindent Hence, $\left|X\right|\geq \omega$ (as defined in (\ref{eqn:cardinalityabb:3})) aligns extensionally with $\left|I_{\omega}\right|\leq \left|X\right|$ as defined in~(\ref{eqn:cardinalityabb:2}).

Sometimes in what follows we use some standard terminology for describing equivalence relations on a set. So suppose that $E$ is an equivalence relation on a set $P$. Usually in what follows $P$ will be the power set $P(M)$ of some set $M$. Then we use $[X]_E = \{Y\in P: E(X,Y)\}$ as an abbreviation for the $E$-equivalence classes of an element~$X$ of $P$. Further, we use $\nicefrac{P}{E} = \{[X]_E: X\in P\}$ for the set of all equivalence classes.  Finally, representatives for the equivalence classes will be given by any injection $\iota: \nicefrac{P}{E}\rightarrow P$ such that $E(X,\iota([X]_E))$. More generally, often in what follows we shall be interested in the related situation of injections $\iota: \nicefrac{P(M)}{E}\rightarrow M$. As we'll see at the outset of the next section, this is sometimes useful for construction models of abstraction principles. 

\section{The Natural Bijection}\label{sec02.5}

With these preliminary definitions pertaining to the background second-order logic in place, we can now proceed to define what a model of an abstraction principle is. Suppose that~$E(X,Y)$ is an~$L_0$-formula with exactly two free concept variables. Let~$L_0[\partial]$ be the expansion of~$L_0$ by a function symbol~$\partial$ which takes unary concepts as inputs and outputs objects. Then the \emph{abstraction principle~$A[E]$ associated to~$E$} is as in equation~(\ref{eqn:AE}) from the previous section. By abuse of notation, we also use $A[E]$ to refer to the theory consisting of this abstraction principle, the full comprehension schema~(\ref{eqn:fullcomp}), the axiom of choice ${\tt AC}$~(\ref{eqn:AC}), and the global choice schema ${\tt GC}$~(\ref{eqn:gcschema}). Then in contrast to equation~(\ref{eqn:defn:models2}), models of~$A[E]$ have the  form:
\begin{myequation}\label{eqn:ohsaycanyouseee}
\mathcal{M}=(M,S_1[M], S_2[M], \ldots, <,\partial)
\end{myequation}
wherein~$\partial:S_1[M]\rightarrow M$ and wherein~$\mathcal{M}$ models the full comprehension schema~(\ref{eqn:fullcomp}), the axiom of choice~(\ref{eqn:AC}), and the global choice schema~${\tt GC}$~(\ref{eqn:gcschema}). In the case where the structure~$\mathcal{M}$ from~(\ref{eqn:defn:models2}) is standard, if there is an injection $\iota:\nicefrac{P(M)}{E}\rightarrow M$, then one can build a model as in equation~(\ref{eqn:ohsaycanyouseee}) by setting $\partial(X)=\iota([X]_E)$. Of course, there is such an injection if and only if $\left|\nicefrac{P(M)}{E}\right|\leq \left|M\right|$, which is a non-trivial assumption. 

Now consider the case in which there are, within a single model, two abstraction operators which satisfy a given abstraction principle. Again suppose that~$E(X,Y)$ is an~$L_0$-formula with exactly two free concept variables. Let~$L_0[\partial_1, \partial_2]$ be the expansion of~$L_0$ by two function symbols~$\partial_1, \partial_2$ which takes unary concepts as inputs and outputs objects. Then the theory $A^2[E]$ consists of the axiom~(\ref{eqn:AE2}) from \S\ref{sec01}, 
as well as the the full comprehension schema~(\ref{eqn:fullcomp}), the axiom of choice~${\tt AC}$~(\ref{eqn:AC}), and the global choice schema~${\tt GC}$~(\ref{eqn:gcschema}). So models of~$A^2[E]$ have the following form:
\begin{myequation}\label{eqn:1}
\mathcal{M}=(M, S_1[M], S_2[M], \ldots,<, \partial_1, \partial_2)
\end{myequation}
wherein~$\partial_i:S_1[M]\rightarrow M$ and $\mathcal{M}$ models the full comprehension schema~(\ref{eqn:fullcomp}), the axiom of choice~${\tt AC}$~(\ref{eqn:AC}), and the global choice schema~${\tt GC}$~(\ref{eqn:gcschema}).

In the description of natural relative categoricity from the earlier section, one of the key ideas is that we restrict down to the ranges of the individual abstraction operators. Formally, we make this precise by taking a model~$\mathcal{M}$ of~$A^2[E]$ as in equation~(\ref{eqn:1}), and defining the following induced~$L_0[\partial_i]$-structure for $i=1,2$:
\begin{myequation}\label{eqn:induced}
\mathcal{M}_i = (\mathrm{rng}(\partial_i), S_1[M]\cap P(\mathrm{rng}(\partial_i)), S_2[M]\cap P(\mathrm{rng}(\partial_i)^2), \ldots, \partial_i\upharpoonright (S_1[M]\cap P(\mathrm{rng}(\partial_i))))
\end{myequation}
Hence $\mathcal{M}\mapsto \mathcal{M}_1$ and $\mathcal{M}\mapsto \mathcal{M}_2$ are maps from an $L_0[\partial_1, \partial_2]$-structure $\mathcal{M}$ to an $L_0[\partial_i]$-structure $\mathcal{M}_i$. So notationally, $\mathcal{M}_i$ is a structure induced from $\mathcal{M}$, and not simply yet another structure indexed by a subscript. Note that we do \emph{not} include the global-well order $<$ in the signature of the induced structures $\mathcal{M}_i$. This is because our natural relative categoricity concerns isomorphisms between these structures, and we do not want to insist that isomorphisms preserve this global well-order. This 
is because the global well-order is an artifact employed to make various second-order notions like cardinality more like classical metatheoretic notions. But of course since we're reasoning about the induced structures $\mathcal{M}_i$ as defined within the larger structure $\mathcal{M}$, we can use global choice in that setting to reason about the induced structures if we like. This disparity between the induced structures $\mathcal{M}_i$ and the structure $\mathcal{M}$ as regards global choice does not extend to the issue of comprehension. For, the induced structures $\mathcal{M}_i$ models the full comprehension schema~(\ref{eqn:fullcomp}) in its signature simply because they are definable within the structure $\mathcal{M}$ which is assumed to satisfy comprehension in its signature. However, note that in general there is no reason that the induced structure~$\mathcal{M}_i$ need model the abstraction principle~$A[E]$.

The notion of natural relative categoricity was defined in \S\ref{sec01} by the condition that the map $\Gamma(\partial_1(X))=\partial_2(X)$ was an isomorphism. Before further examining the condition that this map is an isomorphism, let us take a first and preliminary step of examining the properties of the map itself, which we call the natural bijection. So given any model~$\mathcal{M}$ of~$A^2[E]$ as in equation~(\ref{eqn:1}), \emph{the natural bijection} $\Gamma$ is the map $\Gamma:\mathrm{rng}(\partial_1)\rightarrow \mathrm{rng}(\partial_2)$ defined by 
\begin{myequation}\label{eqn:defn}
\Gamma(\partial_1(X))=\partial_2(X)
\end{myequation}
wherein~$X$ ranges over elements of~$S_1[M]$. It then follows from the axiom~(\ref{eqn:AE2}) of~$A^2[E]$ that this map is well-defined and injective:
\begin{myequation}
\partial_1(X)=\partial_1(Y) \Longleftrightarrow  E(X,Y) \Longleftrightarrow \partial_2(X)=\partial_2(Y)
\end{myequation}
Trivially, by definition,~$\Gamma: \mathrm{rng}(\partial_1)\rightarrow \mathrm{rng}(\partial_2)$ is surjective and so it is indeed a bijection.  Further, the natural bijection is definable in~$\mathcal{M}$ by the following formula:
\begin{myequation}\label{eqn:defn:natret}
\Gamma(x)=y \Longleftrightarrow \mathcal{M}\models [\exists \; X \; (\partial_1(X)=x \; \& \; \partial_2(X)=y)]
\end{myequation}
So by full comprehension~(\ref{eqn:fullcomp}), the graph of the natural bijection~$\Gamma$ is a member of~$S_2[M]$. Likewise, by full comprehension~(\ref{eqn:fullcomp}), if~$X\in S_1[M]\cap P(\mathrm{rng}(\partial_1))$ then the following is an element of $S_1[M]\cap P(\mathrm{rng}(\partial_2))$:
\begin{myequation}\label{eqn:dafasdfsd}
\overline{\Gamma}(X)=\{\Gamma(x): x\in X\}
\end{myequation}
and similarly for~$n$-ary relations. By abuse of notation, we also use the symbol $\overline{\Gamma}$ to refer to the map~$\overline{\Gamma}:\mathcal{M}_1\rightarrow \mathcal{M}_2$ given by~$\Gamma$ on the objects and~$\overline{\Gamma}$ as in equation~(\ref{eqn:dafasdfsd}) on the~$n$-ary relations for all~$n\geq 1$; and we sometimes also refer to the map $\overline{\Gamma}:\mathcal{M}_1\rightarrow \mathcal{M}_2$ as the natural bijection. This map~$\overline{\Gamma}:\mathcal{M}_1\rightarrow \mathcal{M}_2$ is also trivially an injection since~$\Gamma$ is. As is easily verified, it is a surjection as well. 

So, indeed ~$\overline{\Gamma}:\mathcal{M}_1\rightarrow \mathcal{M}_2$ is also a bijection. For ease of future reference, let's record this in the following definition:
\begin{defn}\label{defn:nat:bijection}
Suppose that $\mathcal{M}$ is a model of $A^2[E]$ as in equation~(\ref{eqn:1}), and that $\mathcal{M}_1, \mathcal{M}_2$ are the induced structures as in equation~(\ref{eqn:induced}). Then the \emph{natural bijection} $\overline{\Gamma}:\mathcal{M}_1\rightarrow \mathcal{M}_2$ is given by the bijection $\Gamma:\mathrm{rng}(\partial_1)\rightarrow \mathrm{rng}(\partial_2)$ defined by $\Gamma(\partial_1(X))=\partial_2(X)$ for each concept $X$ from the ambient structure~$\mathcal{M}$.   Further, $\overline{\Gamma}:\mathcal{M}_1\rightarrow \mathcal{M}_2$ is defined on concepts $X\subseteq \mathrm{rng}(\partial_1)$ from the ambient structure by $\overline{\Gamma}(X)=\{\Gamma(x):x\in X\}$, and similarly for $n$-ary relations.
\end{defn}
\noindent The various appeals to the full comprehension schema~(\ref{eqn:fullcomp}) that we made in the previous paragraph underscore the apparent necessity of the adoption of this schema in the context of the present discussion. For instance, to show that the graph of $\Gamma$ exists as a binary concept, we appealed to its definition in equation~(\ref{eqn:defn:natret}), which is $\Sigma^1_1$. This is precisely the amount of comprehension that one needs to show that Basic Law~V, the abstraction principle of Frege's \emph{Grundgesetze}, is inconsistent (cf. \cite{Walsh2012aa} Proposition 4 p. 1682, Proposition 29 p. 1692). Hence, it seems that studying natural relative categoricity in the context of limited comprehension would not be feasible. Before moving on, it's worth recording one final point in regards to the natural bijection: namely, that a routine argument establishes the following. 

\begin{prop}\label{prop:inverse}
If~$\overline{\Gamma}:\mathcal{M}_1\rightarrow \mathcal{M}_2$ is the natural bijection, then its inverse $\Delta=\Gamma^{-1}$ is the natural bijection~$\overline{\Delta}:\mathcal{M}_2\rightarrow \mathcal{M}_1$.
\end{prop}

So having defined the natural bijection $\overline{\Gamma}:\mathcal{M}_1\rightarrow \mathcal{M}_2$, let's now examine carefully what it would mean for this to be a isomorphism between the induced structures $\mathcal{M}_1$ and $\mathcal{M}_2$. Recall that if~$L$ is an arbitrary signature, then two~$L$-structures~$\mathcal{N}_1$ and~$\mathcal{N}_2$ are \emph{isomorphic} if there is a bijection~$\gamma:\mathcal{N}_1\rightarrow \mathcal{N}_2$ such that for all~$L$-formulas~$\theta(x_1, \ldots, x_n)$ and~$a_1, \ldots, a_n$ from~$\mathcal{N}_1$, it is the case that
\begin{myequation}\label{eqn:Iso}
\mathcal{N}_1 \models \theta(a_1, \ldots, a_n) \Longleftrightarrow \mathcal{N}_2 \models \theta(\gamma(a_1), \ldots, \gamma(a_n))
\end{myequation}
Of course, this condition is difficult to verify directly, so one usually works with the equivalent condition that equation~(\ref{eqn:Iso}) holds in the case of atomic formulas (cf. \cite{Marker2002} Definition~1.1.3 pp.~8-9 and the proof of Theorem~1.1.10 p. 13, or \cite{Enderton2001} p.~94 and the Homomorphism Theorem part~(c) p.~96). For instance, consider the atomic formula $\theta(x_1, \ldots, x_n,y)\equiv F(x_1, \ldots, x_n)=y$. Suppose that $a_1, \ldots, a_n, b$ are from $\mathcal{N}_1$ and that $F^{\mathcal{N}_1}(a_1, \ldots, a_n)=b$, so that $\mathcal{N}_1\models \theta(a_1, \ldots, a_n,b)$. Then equation~(\ref{eqn:Iso}) implies that $\mathcal{N}_2\models \theta(\gamma(a_1), \ldots, \gamma(a_n),\gamma(b))$, or that $F^{\mathcal{N}_2}(\gamma(a_1), \ldots, \gamma(a_n))=\gamma(b)$, which of course implies that $F^{\mathcal{N}_2}(\gamma(a_1), \ldots, \gamma(a_n))=\gamma(F^{\mathcal{N}_1}(a_1, \ldots, a_n))$. Elementary considerations such as these show that $\gamma:\mathcal{N}_1\rightarrow \mathcal{N}_2$ is an isomorphism if and only if for all relations symbols $R$, constant symbols $c$, and function symbols $F$ in the signature of the structures, and all $a_1, \ldots, a_n$ from $\mathcal{N}_1$, one has
\begin{myeqnarray}
R^{\mathcal{N}_1}(a_1, \ldots, a_n) & \Longleftrightarrow &  R^{\mathcal{N}_2}(\gamma(a_1), \ldots, \gamma(a_n)) \label{eqn:defn:iso:rel} \\
\gamma(c^{\mathcal{N}_1})&  = & c^{\mathcal{N}_2} \label{eqn:defn:iso:con} \\
\gamma(F^{\mathcal{N}_1}(a_1, \ldots, a_n)) & = &  F^{\mathcal{N}_2}(\gamma(a_1), \ldots, \gamma(a_n)) \label{eqn:defn:iso:func}
\end{myeqnarray}
\noindent While these considerations are admittedly elementary, it's worth underscoring them since they help to motivate the definition of natural relative categoricity (Definition~\ref{eqn:defn:NRC}), which we now build towards.

Now consider the natural bijection~$\overline{\Gamma}:\mathcal{M}_1\rightarrow \mathcal{M}_2$ and what it would mean for it to be an isomorphism. Since $\overline{\Gamma}(X)=\{\Gamma(x): x\in X\}$, clearly one has that equation~(\ref{eqn:defn:iso:rel}) always holds in the case of the predication relations, which per the definition of $L_0$ in Definition~\ref{defn:L000} are the \emph{only} relations in the signature of $L_0$. So $\overline{\Gamma}:\mathcal{M}_1\rightarrow \mathcal{M}_2$ is an isomorphism of~$L_0[\partial]$-structures if and only if equation~(\ref{eqn:defn:iso:func}) holds with respect to the operator~$\partial$. That is, the natural bijection ~$\overline{\Gamma}:\mathcal{M}_1\rightarrow \mathcal{M}_2$ is an isomorphism of~$L_0[\partial]$-structures if and only if
\begin{myequation}
X\in (S_1[M]\cap P(\mathrm{rng}(\partial_1))\Longrightarrow \Gamma(\partial_1(X))=\partial_2(\overline{\Gamma}(X))
\end{myequation}
Given the way that the natural bijection~$\Gamma$ was defined in equation~(\ref{eqn:defn}), this happens if and only if 
\begin{myequation}
X\in (S_1[M]\cap P(\mathrm{rng}(\partial_1))\Longrightarrow \partial_2(X)=\partial_2(\overline{\Gamma}(X))
\end{myequation}
which, by the fact that the model~$\mathcal{M}$ from equation~(\ref{eqn:1}) satisfies~$A^2[E]$, holds if and only if 
\begin{myequation}\label{eqn:final}
\mathcal{M} \models [\forall \; X \; (X\subseteq \mathrm{rng}(\partial_1)) \rightarrow E(X, \overline{\Gamma}(X))]
\end{myequation}
Now, the natural bijection $\overline{\Gamma}: \mathcal{M}_1\rightarrow \mathcal{M}_2$ is an isomorphism if and only if its inverse $\Delta=\Gamma^{-1}$ is an isomorphism, and by Proposition~\ref{prop:inverse} its inverse is the natural bijection $\overline{\Delta}:\mathcal{M}_2\rightarrow \mathcal{M}_1$. Hence, by parity of reasoning and the fact that $E$ is an equivalence relation, one has that $\overline{\Gamma}: \mathcal{M}_1\rightarrow \mathcal{M}_2$ is an isomorphism if and only if  
\begin{myequation}\label{eqn:final2}
\mathcal{M} \models [\forall \; Y \; (Y\subseteq \mathrm{rng}(\partial_2)) \rightarrow E(Y, \overline{\Gamma}^{-1}(Y))]
\end{myequation}

For ease of future reference, let us summarize these results as follows. First let's record our official definition of natural relative categoricity:
\begin{defn}\label{eqn:defn:NRC}
An abstraction principle $A[E]$ is \emph{naturally relatively categorical} if all models~$\mathcal{M}$ of~$A^2[E]$ from equation~(\ref{eqn:1}), the natural bijection $\overline{\Gamma}:\mathcal{M}_1\rightarrow \mathcal{M}_2$ from Definition~\ref{defn:nat:bijection} is an isomorphism of the induced structures $\mathcal{M}_1, \mathcal{M}_2$ from equation~(\ref{eqn:induced}).
\end{defn}
\noindent The only way in which this formalization of the notion is more precise than the descriptions of this notion given in \S\ref{sec01} is that now we have formally defined the particulars of our background second-order logic and have likewise defined the natural bijection and indicated precisely what it takes for it to be a isomorphism. 

The elementary considerations from the previous paragraphs give us a simple equivalent characterization of natural relative categoricity. In particular, we have:
\begin{prop}\label{prop:simpleequivalent} An abstraction principle $A[E]$ is naturally relatively categorical if and only if all models~$\mathcal{M}$ of~$A^2[E]$ from equation~(\ref{eqn:1}) satisfy one of the two following equivalent conditions:
\begin{myequation} 
\mathcal{M} \models [\forall \; X \; (X\subseteq \mathrm{rng}(\partial_1)) \rightarrow E(X, \overline{\Gamma}(X))] \tag{\ref{eqn:final}}
\end{myequation}\vspace{-10mm}
\begin{myequation}
\mathcal{M} \models [\forall \; Y \; (Y\subseteq \mathrm{rng}(\partial_2)) \rightarrow E(Y, \overline{\Gamma}^{-1}(Y))]\tag{\ref{eqn:final2}}
\end{myequation}
wherein $\Gamma$ is the natural bijection (cf. Definition~\ref{defn:nat:bijection}).
\end{prop}
\noindent Expressed in these terms, natural relative categoricity is patently a deductive property of the theory~$A^2[E]$.

In the next section we'll prove Theorem~\ref{thm:ncr=iia=cca} which gives a characterization of natural relative categoricity in terms of cardinality coarsening on abstracts~(\ref{eqn:defn:CC}) and injection invariance on abstracts~(\ref{eqn:defn:II}). In the previous section we've formally defined our background second-order logic and so we can be a bit more precise now about the content of these conditions. Recall that $A[E]$ may be used as the abbreviation for the theory consisting of the abstraction principle~(\ref{eqn:AE}) in addition to the full comprehension schema~(\ref{eqn:fullcomp}), the axiom of choice~(\ref{eqn:AC}), and the global choice schema~(\ref{eqn:gcschema}). Then officially, we say that $A[E]$ is \emph{cardinality coarsening on abstracts} if the following is a theorem of $A[E]$:
\begin{myequation}
\forall \; X, Y\; ((Y\approx X \; \& \; X\subseteq \mathrm{rng}(\partial)) \rightarrow E(X,Y))]\tag{\ref{eqn:defn:CC}}
\end{myequation}
Likewise, officially $A[E]$ is \emph{injection invariant on abstracts} if the following is a theorem of $A[E]$:
\begin{myequation}
\forall \; \mbox{ injection } \iota:V\hspace{-1mm}\rightarrow\hspace{-1mm}V \; \forall \; X\subseteq \mathrm{rng}(\partial) \;E(X,\overline{\iota}(X))\tag{\ref{eqn:defn:II}}
\end{myequation} 
So both cardinality coarsening on abstracts and injection invariance on abstracts are, by definition, deductive properties of the theory $A[E]$. By contrast, as was made clear by Proposition~\ref{prop:simpleequivalent} of the previous paragraph, natural relative categoricity is a deductive property of the theory $A^2[E]$. So one of the implications of Theorem~\ref{thm:ncr=iia=cca}, which establishes the equivalence of these notions, is that we're able to further reduce natural relative categoricity to a deductive property of the theory $A[E]$ as opposed to $A^2[E]$.

Before setting up the particulars of our second main theorem, let's record for reference when an isomorphism $\overline{H}:\mathcal{M}_1\rightarrow \mathcal{M}_2$ is equal to the natural bijection $\overline{\Gamma}: \mathcal{M}_1\rightarrow \mathcal{M}_2$. It's natural to focus attention on those isomorphisms $\overline{H}:\mathcal{M}_1\rightarrow \mathcal{M}_2$ whose restriction $H\upharpoonright \mathrm{rng}(\partial_1)$ is an element of $S_2[M]$, since it is only with respect to these that we can define further elements of $\mathcal{M}$ in terms of $\overline{H}$ and $H$ by recourse to the comprehension schema~(\ref{eqn:fullcomp}). By a routine argument, we can establish the following: 

\begin{prop}\label{prop:whenarbirisnat}
Suppose that $\mathcal{M}$ is a model of $A^2[E]$ with induced structures $\mathcal{M}_1$ and $\mathcal{M}_2$. Suppose that $\overline{H}:\mathcal{M}_1\rightarrow \mathcal{M}_2$ is a map whose restriction $H\upharpoonright \mathrm{rng}(\partial_1)$ to $\mathrm{rng}(\partial_1)$ is an element of $S_2[M]$. Then $\overline{H}:\mathcal{M}_1\rightarrow \mathcal{M}_2$ is isomorphism if and only if one has $\mathcal{M}\models [\forall \; X \; (X\subseteq \mathrm{rng}(\partial_1)\rightarrow E(X, \overline{H}(X))]$.
\end{prop}
\noindent As a corollary to Theorem~\ref{thm:ncr=iia=cca}, we will establish in the next section that such isomorphisms are \emph{always} equal to the natural bijection in the setting of natural relative categoricity (cf. Corollary~\ref{cor:whenarbirisnat}).

Finally, let's briefly say something about the content of our second main Theorem~\ref{thm:coverthm}. For ease of future reference, let's record the official notion in the following definition:
\begin{defn}\label{eqn:defn:NRC:covered}
An abstraction principle $A[E]$ is \emph{surjectively relatively categorical} if all models~$\mathcal{M}$ of~$A^2[E]$ from~(\ref{eqn:1}) wherein the abstraction operators $\partial_i:S_1[M]\rightarrow M$ are surjective, the natural bijection $\overline{\Gamma}:\mathcal{M}_1\rightarrow \mathcal{M}_2$ from Definition~\ref{defn:nat:bijection} is an isomorpism of the induced structures $\mathcal{M}_1, \mathcal{M}_2$ from equation~(\ref{eqn:induced}).
\end{defn}
\noindent So clearly natural relative categoricity implies surjective relative categoricity. For an example of an abstraction principle which is surjectively relatively categorical but not naturally relatively categorical, see the example of the Bicardinality Principle in \S\ref{sec:bicardinality}. In the statement of Theorem~\ref{thm:coverthm}, the key notions were that of permutation invariance~(\ref{eqn:defn:piI}) and bicardinality coarsening~(\ref{eqn:defn:biC}). In the context of Theorem~\ref{thm:coverthm}, it is understood that to say an abstraction principle $A[E]$ has one of these properties is to say that these properties are deducible from the supposition that (i) the abstraction operator is a surjection as well as from (ii) the abstraction principle itself, the full comprehension schema~(\ref{eqn:fullcomp}), the axiom of choice~(\ref{eqn:AC}), and the global choice schema~${\tt GC}$~(\ref{eqn:gcschema}). 

\section{The Equivalent Characterizations}\label{sec03}

The goal of this section is to establish Theorem~\ref{thm:ncr=iia=cca} and Theorem~\ref{thm:coverthm}. The first provides an equivalent characterization of natural relative categoricity in terms of cardinality coarsening on abstracts~(\ref{eqn:defn:CC}) and injective invariance on abstracts~(\ref{eqn:defn:II}). The second provides an equivalent characterization of surjectively relatively categoricity in terms of bicardinality coarsening~(\ref{eqn:defn:biC}) and permutation invariance~(\ref{eqn:defn:piI}). These theorems gives us two qualitatively distinct means by which to identify and recognize our versions of relative categoricity. For, in and of themselves, natural relative categoricity and surjective relative categoricity  (Definition~\ref{eqn:defn:NRC} and  Definition~\ref{eqn:defn:NRC:covered}) are claims about determining a single structure. But the notions of injection  invariance on abstracts~(\ref{eqn:defn:II}) and permutation invariance~(\ref{eqn:defn:piI}) deal in a different currency:  these conditions say say that a certain second-order relation should be invariant under certain mappings of the entire domain.  Cardinality coarsening on abstracts~(\ref{eqn:defn:CC}) and bicardinality coarsening~(\ref{eqn:defn:biC}) are different still: they are more local in character and concern the comparative sizes of a concept (and its relative complement).

Here is then the proof of Theorem~\ref{thm:ncr=iia=cca}:
\begin{proof}
First suppose that $A[E]$ is naturally relatively categorical. Suppose that~$\mathcal{M}^{\ast}=(M, S_1[M], S_2[M], \ldots, <, \partial_1)$ is an arbitrary model of~$A[E]$. Note that~$V=\{x:x=x\}$~(\ref{eqn:defnV}) as interpreted on~$\mathcal{M}^{\ast}$ is exactly~$M$. So with an eye towards showing injection invariance on abstracts, suppose that~$\iota:M\rightarrow M$ is an injection whose graph is in~$S_2[M]$. Then define~$\partial_2:S_1[M]\rightarrow M$ by~$\partial_2=\iota\circ \partial_1$. Note that since~$\iota:M\rightarrow M$ is an injection, we have that the following holds in $\mathcal{M}$:
\begin{myequation}\label{eqn:iamanaequationinaproof}
\partial_2(X)=\partial_2(Y) \Longleftrightarrow \partial_1(X)=\partial_1(Y) \Longleftrightarrow E(X,Y)
\end{myequation}
Hence, since~$\partial_2$ is~$\mathcal{M}^{\ast}$-definable, the following structure is a model of~$A^2[E]$:
\begin{myequation}
 \mathcal{M}=(M, S_1[M], S_2[M], \ldots, <,\partial_1, \partial_2)
\end{myequation}
Since by hypothesis~$A[E]$ is naturally relatively categorical, we have that the natural bijection~$\overline{\Gamma}: \mathcal{M}_1\rightarrow \mathcal{M}_2$ is an isomorphism. Then by definition of~$\Gamma$ in equation~(\ref{eqn:defn}), we have that $X\in S_1[M]$ implies $\Gamma(\partial_1(X)) = \partial_2(X) = \iota (\partial_1(X))$. Hence~$\Gamma=\iota \upharpoonright (\mathrm{rng}(\partial_1))$. Now we may finally finish verifying injection invariance on abstracts. Suppose that~$X\in S_1[M]\cap P(\mathrm{rng}(\partial_1))$. Then~$\overline{\Gamma}(X)=\overline{\iota}(X)$. Hence from Proposition~\ref{prop:simpleequivalent} (and in particular equation~(\ref{eqn:final})) we may infer that~$E(X,\overline{\Gamma}(X))$ and hence~$E(X,\overline{\iota}(X))$.

Second suppose~$A[E]$ is injection invariant on abstracts. To show that~$A[E]$ is naturally relatively categorical, suppose that 
\begin{myequation}
\mathcal{M} = (M, S_1[M], S_2[M], \ldots, <,\partial_1, \partial_2)
\end{myequation}
is a model of~$A^2[E]$. So we must show that the natural bijection~$\Gamma:\mathcal{M}_1\rightarrow \mathcal{M}_2$ is an isomorphism. By Cardinal Comparability~(\ref{eqn:CC}), 
\begin{myequation}\label{eqn:choice}
\left|M\setminus \mathrm{rng}(\partial_1)\right|\leq  \left| M\setminus \mathrm{rng}(\partial_2)\right| \hspace{5mm} \mbox{or} \hspace{5mm} \left| M\setminus \mathrm{rng}(\partial_2)\right| \leq \left| M\setminus \mathrm{rng}(\partial_1)\right|
\end{myequation}
First suppose that $\left|M\setminus \mathrm{rng}(\partial_1)\right|\leq  \left| M\setminus \mathrm{rng}(\partial_2)\right|$. Then let~$\Delta:M\setminus \mathrm{rng}(\partial_1) \rightarrow M\setminus \mathrm{rng}(\partial_2)$ be a witnessing injection. Define an injection~$\iota:M\rightarrow M$ by
\begin{myequation}
\iota \upharpoonright \mathrm{rng}(\partial_1) = \Gamma, \hspace{10mm}\iota \upharpoonright (M\setminus \mathrm{rng}(\partial_1))=\Delta
\end{myequation}
Since~$\Gamma$ has range~$\mathrm{rng}(\partial_2)$ and~$\Delta$ has range~$M\setminus \mathrm{rng}(\partial_2)$, the map~$\iota:M\rightarrow M$ is indeed an injection. Now, we verify natural relative categoricity by verifying equation~(\ref{eqn:final}). So suppose that~$X\in (S_1[M]\cap P(\mathrm{rng}(\partial_1))$. Then by injection invariance on abstracts  applied to~$\iota$, we have that~$\mathcal{M}\models E(X,\overline{\iota}(X))$. But since~$X\subseteq \mathrm{rng}(\partial_1)$, we have that~$\overline{\iota}(X)=\overline{\Gamma}(X)$, so that~$\mathcal{M}\models E(X,\overline{\Gamma}(X))$. Hence, we have finished verifying natural relative categoricity via equation~(\ref{eqn:final}).

Conversely, suppose that~$\left|M\setminus \mathrm{rng}(\partial_2)\right|\leq \left|M\setminus \mathrm{rng}(\partial_1)\right|$ with witnessing injection~$\Delta:M\setminus \mathrm{rng}(\partial_2) \rightarrow M\setminus \mathrm{rng}(\partial_1)$. Define an injection~$\iota:M\rightarrow M$ by
\begin{myequation}
\iota \upharpoonright \mathrm{rng}(\partial_2) = \Gamma^{-1}, \hspace{10mm}\iota \upharpoonright (M\setminus \mathrm{rng}(\partial_2))=\Delta
\end{myequation}
Since~$\Gamma^{-1}$ has range~$\mathrm{rng}(\partial_1)$ and~$\Delta$ has range~$M\setminus \mathrm{rng}(\partial_1)$, the map~$\iota:M\rightarrow M$ is indeed an injection. Now, we verify natural relative categoricity by verifying equation~(\ref{eqn:final2}). So suppose that~$Y\in (S_1[M]\cap P(\mathrm{rng}(\partial_2))$. Then by injection invariance on abstracts applied to~$\iota$, we have that~$\mathcal{M}\models E(Y,\overline{\iota}(Y))$. But since~$Y\subseteq \mathrm{rng}(\partial_2)$, we have that~$\overline{\iota}(Y)=\overline{\Gamma}^{-1}(Y)$, so that~$\mathcal{M}\models E(Y,\overline{\Gamma}^{-1}(Y))$. Hence, we have finished verifying natural relative categoricity via equation~(\ref{eqn:final2}).

Having shown the equivalence of natural relative categoricity and injection invariance on abstracts, we now show that these are equivalent to cardinality coarsening on abstracts. First, note that 
cardinality coarsening on abstracts trivially implies injection invariance on abstracts. For, suppose that we're working in a model of~$A[E]$ and there's an injection~$\iota:V\rightarrow V$ and~$X\subseteq \mathrm{rng}(\partial_1)$. Then let~$Y=\overline{\iota}(X)$, so that~$Y\approx X$ and $X \subseteq \mathrm{rng}(\partial_1)$. Then by  cardinality coarsening on abstracts, we have~$E(X,Y)$, which is just to say~$E(X,\overline{\iota}(X))$, so that we have verified injection invariance on abstracts.

Now assume that~$A[E]$ is injection invariant on abstracts. Consider a model~$\mathcal{M}^{\ast}=(M, S_1[M], S_2[M], \ldots,<, \partial_1)$ of~$A[E]$, and suppose that~$X_0,Y_0$ are members of~$S_1[M]$ with~$Y_0\approx X_0$ and $X_0\subseteq \mathrm{rng}(\partial_1)$. We must show that~$E(X_0,Y_0)$. There are several cases to consider, which for the sake of readability, we enumerate separately. For the abbreviations of cardinality notions which we employ here, see in particular (\ref{eqn:cardinalityabb:1})-(\ref{eqn:cardinalityabb:3}) from \S\ref{sec02}.

Case I:~$\left| Y_0\right| = \left| X_0\right| =\left| V\right|$. Then~$V\approx X_0\subseteq \mathrm{rng}(\partial_1)\subseteq V$ implies that~$\left|\mathrm{rng}(\partial_1)\right| = \left|V\right|$. Choose a bijection~$\pi: \mathrm{rng}(\partial_1)\rightarrow V$ and define a map~$\partial_2:S_1[M]\rightarrow M$ by~$\partial_2(X)=\pi(\partial_1(X))$, so that since~$\pi: \mathrm{rng}(\partial_1)\rightarrow V$ is a surjection, we have that~$\mathrm{rng}(\partial_2)=V$. Then observe that for any~$X,Y$ we have the following in $\mathcal{M}^{\ast}$ since~$\pi$ is an injection:
\begin{myequation}\label{eqn:iamausualmove0}
\partial_2(X)=\partial_2(Y) \Longleftrightarrow \partial_1(X)=\partial_1(Y) \Longleftrightarrow E(X,Y)
\end{myequation}
Hence, the structure~$\mathcal{M}^{\dagger}=(M, S_1[M], S_2[M], \ldots, <,\partial_2)$ is likewise a model of~$A[E]$, and so by injection invariance on abstracts we have
\begin{myequation}
\mathcal{M}^{\dagger}\models  [ \forall \; \mbox{ injection } \iota:V\hspace{-1mm}\rightarrow\hspace{-1mm}V \; \forall \; X\subseteq \mathrm{rng}(\partial_2) \;E(X,\overline{\iota}(X))]
\end{myequation}
But since~$\mathrm{rng}(\partial_2)=V$, this can be simplified to: 
\begin{myequation}
\mathcal{M}^{\dagger}\models  [ \forall \; \mbox{ injection  }\iota:V\hspace{-1mm}\rightarrow\hspace{-1mm}V \; \forall \; X \;E(X,\overline{\iota}(X))]
\end{myequation}
Now, since~$Y_0\approx X_0\approx V$, choose bijections~$j_1:V\rightarrow X_0$ and~$j_2:V\rightarrow Y_0$. Then by the previous equation, we have~$E(V,\overline{j_1}(V))$ and~$E(V,\overline{j_2}(V))$, or what is the same~$E(V,X_0)$ and~$E(V,Y_0)$. Since $E$ is an equivalence relation, we have that~$E(X_0,Y_0)$, which was to be demonstrated.

Case II:~$\left|Y_0\right|= \left|X_0\right| <\left|V\right| \; \& \: \left|V\right|<\omega$. By Cardinal Comparability~(\ref{eqn:CC}), 
\begin{myequation}
\left|V\setminus X_0 \right|\leq \left|V\setminus Y_0 \right|\hspace{5mm} \mbox{or} \hspace{5mm} \left|V\setminus Y_0 \right|\leq \left|V\setminus X_0 \right| 
\end{myequation}
First suppose that~$\left|V\setminus X_0 \right|\leq \left|V\setminus Y_0 \right|$. Choose injection~$\iota:V\rightarrow V$ such that~$\iota\upharpoonright X_0: X_0\rightarrow Y_0$ is a bijection and~$\iota \upharpoonright (V\setminus X_0): V\setminus X_0 \rightarrow V\setminus Y_0$ is an injection. Then since~$A[E]$ is injection invariant on abstracts, we have that~$E(X_0,\overline{\iota}(X_0))$, which is the same as~$E(X_0,Y_0)$. Second suppose that~$\left|V\setminus Y_0 \right|\leq \left|V\setminus X_0 \right|$. Choose injection~$j:V\rightarrow V$ such that~$j\upharpoonright Y_0: Y_0\rightarrow X_0$ is a bijection and~$j \upharpoonright (V\setminus Y_0): V\setminus Y_0 \rightarrow V\setminus X_0$ is an injection. Since~$\left|V\right|<\omega$, we have that the injection ~$j:V\rightarrow V$ is actually a bijection. Let~$\iota = j^{-1}$, so that~$\iota:V\rightarrow V$ is a bijection and~$\overline{\iota}(X_0)=Y_0$. Then by injection invariance on abstracts, we again have~$E(X_0, \overline{\iota}(X_0))$ or~$E(X_0, Y_0)$.

Case III:~$\left|Y_0\right|= \left|X_0\right| <\left|V\right| \; \& \: \left|V\right|\geq \omega$. Then by Infinite Sums are Maxs~(\ref{eqn:InfSumMax}), our case assumptions imply that $\left|V\right|  =   \max\{\left|V\setminus X_0\right|, \left|X_0\right|\} = \left|V\setminus X_0\right|$ and $\left|V\right|   =  \max\{\left|V\setminus Y_0\right|, \left|Y_0\right|\} = \left|V\setminus Y_0\right|$. 
Hence, trivially one has that~$ \left|V\setminus X_0\right|\leq \left|V\setminus Y_0\right|$. Choose injection~$\iota:V\rightarrow V$ such that~$\iota\upharpoonright X_0: X_0\rightarrow Y_0$ is a bijection and~$\iota \upharpoonright (V\setminus X_0): V\setminus X_0 \rightarrow V\setminus Y_0$ is an injection. Then since~$A[E]$ is injection invariant on abstracts, we have that~$E(X_0,\overline{\iota}(X_0))$, which is the same as~$E(X_0,Y_0)$.
\end{proof}

Finally, let's note an instructive corollary to Theorem~\ref{thm:ncr=iia=cca}. This corollary tells us that in the context of natural relative categoricity, the only definable isomorphisms between the induced structures are identical to the natural bijection:
\begin{cor}\label{cor:whenarbirisnat}
Suppose that $A[E]$ is naturally relatively categorical. Suppose that $\mathcal{M}$ is a model of $A^2[E]$ with induced structures $\mathcal{M}_1$ and $\mathcal{M}_2$. Suppose that $\overline{H}:\mathcal{M}_1\rightarrow \mathcal{M}_2$ is a map whose restriction $H\upharpoonright \mathrm{rng}(\partial_1)$ to $\mathrm{rng}(\partial_1)$ is an element of $S_2[M]$. If $\overline{H}:\mathcal{M}_1\rightarrow \mathcal{M}_2$ is an isomorphism, then it is equal to the natural bijection $\overline{\Gamma}: \mathcal{M}_1\rightarrow \mathcal{M}_2$.
\end{cor}
\begin{proof}
So suppose that $\overline{H}:\mathcal{M}_1\rightarrow \mathcal{M}_2$ is an isomorphism. To show that $\overline{H}$ is equal to the natural bijection $\overline{\Gamma}: \mathcal{M}_1\rightarrow \mathcal{M}_2$, it suffices by Proposition~\ref{prop:whenarbirisnat} to show that 
\begin{myequation}
\mathcal{M}\models [\forall \; X \; (X\subseteq \mathrm{rng}(\partial_1)\rightarrow E(X, \overline{H}(X))]
\end{myequation}
Suppose that $X\subseteq \mathrm{rng}(\partial_1)$, and let $Y=\overline{H}(X)$, so that $Y\approx X$ and $X\subseteq \mathrm{rng}(\partial_1)$. Since $A[E]$ is naturally relatively categorical, by Theorem~\ref{thm:ncr=iia=cca} it is cardinality coarsening on abstracts~(\ref{eqn:defn:CC}), from which we can infer $E(X,Y)$, which is the same as $E(X,\overline{H}(X))$, so that indeed the previous equation is satisfied.
\end{proof}

Let's turn now to the proof of our other main theorem, namely, Theorem~\ref{thm:coverthm}:
\begin{proof}
First suppose that $A[E]$ is surjectively relatively categorical (cf. Definition~\ref{eqn:defn:NRC:covered}). Suppose that $(M, S_1[M], S_2[M], \ldots, <,\partial_1)$ is a model of $A[E]$ where $\partial_1:S_1[M]\rightarrow M$ is a surjection, and suppose that $\pi:M\rightarrow M$ is a bijection. Let $\partial_2:S_1[M]\rightarrow M$ be defined by $\partial_2=\pi\circ \partial_1$. Then since $\pi:M\rightarrow M$ is an injection, we have that the following holds in our model:
\begin{myequation}\label{eqn:iamausualmove}
\partial_2(X)=\partial_2(Y) \Longleftrightarrow \partial_1(X)=\partial_1(Y) \Longleftrightarrow E(X,Y)
\end{myequation}
Hence $\mathcal{M}=(M, S_1[M], S_2[M], \ldots, <,\partial_1, \partial_2)$ is a model of $A^2[E]$. Further, since $\pi:M\rightarrow M$ is an surjection, we have that $\partial_2:S_1[M]\rightarrow M$ is a surjection. Then since by the hypothesis of $A[E]$ being surjectively relatively categorical, we have that the natural bijection $\overline{\Gamma}:\mathcal{M}_1\rightarrow \mathcal{M}_2$ is an isomorphism of the induced structures (cf. equation~(\ref{eqn:induced})). Then by surjective relative categoricity (cf. equation~(\ref{eqn:final})), we have that $\mathcal{M}$ models $E(X,\overline{\Gamma}(X))$ for any $X\in S_1[M]$. But for any $X\in S_1[M]$, one may use the surjectivity of the operators to check that $\overline{\pi}(X)=\overline{\Gamma}(X)$, which of course implies that $E(X,\overline{\pi}(X))$, so that we are done.

Now suppose that $A[E]$ plus the surjectivity of the abstraction operator proves permutation invariance~(\ref{eqn:defn:piI}). Then we show that $A[E]$ is surjectively relatively categorical. So suppose that $\mathcal{M}$ is a model of $A^2[E]$ wherein $\partial_1, \partial_2$ are surjective. Then the natural bijection is a bijection $\Gamma:V\rightarrow V$. Then by permutation invariance~(\ref{eqn:defn:piI}), we have that $E(X,\overline{\Gamma}(X))$ for all~$X$. Then by equation~(\ref{eqn:final}), we have that $\overline{\Gamma}:\mathcal{M}_1\rightarrow \mathcal{M}_2$ is an isomorphism. 

So we've shown that surjective relative categoricity is equivalent to permutation invariance, assuming that the abstraction operator is surjective. Now we show that these two conditions are equivalent to the condition of bicardinality coarsening, under the hypothesis that the abstraction operator is surjective. First suppose that that $A[E]$ plus the surjectivity of the abstraction operator proves permutation invariance~(\ref{eqn:defn:piI}). Then suppose that $X, Y$ such that $\left|X\right|=\left|Y\right|$ and  $\left|V\setminus X\right|=\left|V\setminus Y\right|$. Any two witnessing bijections can be conjoined into a bijection $\pi:V\rightarrow V$ such that $\overline{\pi}(X)=Y$ and $\overline{\pi}(V\setminus X)=V\setminus Y$. Then by permutation invariance~(\ref{eqn:defn:piI}), we have that $E(X,\overline{\pi}(X))$ or $E(X,Y)$, which is what we wanted to establish. Finally, suppose that $A[E]$ plus the surjectivity of the abstraction operator proves bicardinality coarsening~(\ref{eqn:defn:biC}). Suppose that $\pi:V\rightarrow V$ is a bijection. Let $X$ be a concept and let $Y=\overline{\pi}(X)$. Then since $\pi$ is a bijection, we have that $\left|X\right|=\left|Y\right|$ and $\left|V\setminus X\right|=\left|V\setminus Y\right|$. So by bicardinality coarsening~(\ref{eqn:defn:biC}), it follows that $E(X,Y)$ or $E(X,\overline{\pi}(X))$.
\end{proof}

In analogue to Corollary~\ref{cor:whenarbirisnat}, we have the following result showing that the only isomorphism is the natural bijection in the setting of surjective relative categoricity. We omit the proof since it is entirely analogous to the proof of the this earlier corollary.
\begin{cor}\label{cor:whenarbirisnat2}
Suppose that $A[E]$ is surjectively relatively categorical. Suppose that $\mathcal{M}$ is a model of $A^2[E]$, where the abstraction operators are surjective,  with induced structures $\mathcal{M}_1$ and $\mathcal{M}_2$. Suppose that $\overline{H}:\mathcal{M}_1\rightarrow \mathcal{M}_2$ is a map whose restriction $\overline{H}\upharpoonright \mathrm{rng}(\partial_1)$ to $\mathrm{rng}(\partial_1)$ is an element of $S_2[M]$. If $\overline{H}:\mathcal{M}_1\rightarrow \mathcal{M}_2$ is an isomorphism, then it is equal to the natural bijection $\overline{\Gamma}: \mathcal{M}_1\rightarrow \mathcal{M}_2$.
\end{cor}

So the equivalent characterizations featuring in Theorem~\ref{thm:ncr=iia=cca} and Theorem~\ref{thm:coverthm} concern the abstraction principle $A[E]$. It will be useful to have analogues of these for the underlying equivalence relations $E$ as well. So we define:
 \begin{myenumerate}
 \item \emph{Injection Invariant}:~$\models [ \forall \; \mbox{ injection } \iota:V\hspace{-1mm}\rightarrow\hspace{-1mm}V \; \forall \; X \;E(X,\overline{\iota}(X))]$\label{eqn:RC2xsemi}
  \item \emph{Permutation Invariant}:~$\models [ \forall \; \mbox{ bijection } \pi:V\hspace{-1mm}\rightarrow\hspace{-1mm}V \; \forall \; X \;E(X,\overline{\pi}(X))]$\label{eqn:RC2xsemi:permu}
\item \emph{Cardinality Coarsening}:~$\models [ \forall \; X, Y\; (\left|X\right|=\left|Y\right|\rightarrow E(X,Y))]$\label{eqn:RC3xsemi}
\item \emph{Bicardinality Coarsening}:~$\models [ \forall \; X, Y\; ((\left|X\right|=\left|Y\right| \; \& \; \left|V\setminus X\right|=\left|V\setminus Y\right|)\rightarrow E(X,Y))]$\label{eqn:RC3xsemi:bicardinality}
\item \emph{Injection Invariant on Small Concepts}: \label{eqn:RC2xsmall}
\item[] \hspace{30mm}~$\models [ \forall \; \mbox{ injection  }\iota:V\hspace{-1mm}\rightarrow\hspace{-1mm}V \; \forall \; X \; (\left|X\right|\leq \left|\nicefrac{P(V)}{E}\right| \;\rightarrow E(X,\overline{\iota}(X)))]$
\item \emph{Cardinality Coarsening on Small Concepts}: \label{eqn:RC3xsmall}
\item[] \hspace{30mm}~$\models [ \forall \; X, Y\; (\left|Y\right|=\left|X\right|\leq \left|\nicefrac{P(V)}{E}\right| \rightarrow E(X,Y))]$
\end{myenumerate}
In the statements of these notions, the~$\models$ relation is the deduction relation from the deduction system of second-order logic which is the background of all our theories. In the final two conditions~(\ref{eqn:RC2xsmall})-(\ref{eqn:RC3xsmall}), the clause~$\left|X\right|\leq \left|\nicefrac{P(V)}{E}\right|$ just means that there's an injection from~$X$ to the~$E$-equivalence classes of second-order objects. This can be written in second-order logic as follows:
\begin{myequation}\label{eqn:whatimeans}
\exists \; R \; \forall \; x,y\in X \; (x\neq y \rightarrow \neg E(R[x], R[y]))
\end{myequation}
wherein~$R$ is a binary relation and where~$R[x]=\{z: Rxz\}$. These conditions~(\ref{eqn:RC2xsemi})-(\ref{eqn:RC3xsmall}) are all conditions on an arbitrary formula~$E(X,Y)$, which we assume to be an $L_0$-formula (cf. Definition~\ref{defn:L000}) which is provably an equivalence relation in our background second-order logic.

In diagrammatic form, the relationship between these notions are displayed in Figure~\ref{diagram1}. In the figure, the arrows going in both directions-- i.e. ``$\leftrightarrow$''-- indicate a provable equivalence, while the double-lined arrows that go in only one direction--  i.e. ``$\Rightarrow$''-- indicate that the implication cannot be reversed. The implications between e.g. injection 
invariance and cardinality coarsening follows automatically from the analogous part of the proof of Theorem~\ref{thm:ncr=iia=cca}. The only implications in Figure~\ref{diagram1} that are less immediate are the implications from the conditions in the middle column to the conditions in the far-right column. So let's show that if $E$ is cardinality coarsening on small concepts then $A[E]$ is cardinality coarsening on abstracts. To see this, it suffices to  note that the Axiom of Choice~${\tt AC}$~(\ref{eqn:AC}) implies that
\begin{myequation}\label{eqn:niceappchoice}
\left|\mathrm{rng}(\partial)\right|\leq \left|\nicefrac{P(V)}{E}\right|
\end{myequation}
For, for each $x$ from $\mathrm{rng}(\partial)$ there is $X$ such that $\partial(X)=x$. Then by ${\tt AC}$~(\ref{eqn:AC}), there is $R$ such that for all $x$ from $\mathrm{rng}(\partial)$ one has that $\partial(R[x])=x$, where $R[x]=\{y: Rxy\}$ which exists by comprehension. Then suppose that $x\neq y$ are both from $\mathrm{rng}(\partial)$ but $E(R[x], R[y])$. Then by $A[E]$ we have $x=\partial(R[x])=\partial(R[y])=y$, a contradiction. So indeed we have~(\ref{eqn:niceappchoice}).

Let us end this section by discussing the  witnessing counterexamples featuring in Figure~\ref{diagram1}. For an example of an $E$ which is bicardinality coarsening but not cardinality coarsening, see the discussion of the Bicardinality Principle in \S\ref{sec:bicardinality}. There are different kinds of witnessing counterexamples which show that the conditions on $A[E]$ don't imply the analogous conditions on $E$. One way to see this is to consider~$A[E]$ which are inconsistent. For instance, take~$E_0(X,Y)\equiv X=Y$. Then $A[E_0]$ is just the abstraction principle Basic Law~V from Frege's \emph{Grundgesetze}, and as we noted earlier this is inconsistent with full comprehension (which we are assuming in this paper). Since the conditions on~$A[E]$ are conditions on derivability in~$A[E]$, trivially the inconsistent~$A[E_0]$ satisfies all of these. One might hope that the right-most arrows could be reversed with the additional assumption of the consistency of~$A[E]$.  But this is not the case. For consider
\begin{myequation}\label{eqn:noreverseconst}
E(X,Y) \equiv (|V|<\omega \; \&\; X = Y) \vee (|V|\geq \omega\; \&\; X \approx Y)
\end{myequation}
Again, for our abbreviations of cardinality-related notions, see  (\ref{eqn:cardinalityabb:1})-(\ref{eqn:cardinalityabb:3}) from \S\ref{sec02}. So on finite domains,~$E$ acts like~$E_0$.  But for~$|V|\geq \omega$,~$E$ is just equinumerosity.
The principle~$A[E]$ is also naturally relatively categorical, since~$A[E]$ only has models with infinite first-order domains, and in those models,~$E$ acts exactly like equinumerosity, which we showed to be naturally relatively categorical in the earlier paper (\cite{Walsh2012aa} Proposition 14 p. 1687; see also \S\ref{sec04}). However,~$E$ is not cardinality coarsening on small concepts (or bicardinality coarsening), since if~$1<|V|<\omega$ and~$a\neq b$ are distinct objects then one has that~$\neg E(\{a\}, \{b\})$.

Second, let's show that cardinality coarsening on small concepts does not imply cardinality coarsening. First define the auxiliary formula:
\begin{myequation}
\theta(X,Y)\equiv \left|X\right|=\left|Y\right| \wedge [(X\neq V \; \& \; Y\neq V) \vee (X=V \; \& \; Y=V)]
\end{myequation}
Consider then the following equivalence relation~$E$:
\begin{myeqnarray}\label{eqn:iamanequivalencerelation}
E(X,Y) & \equiv & \left|V\right|=\omega_1 \rightarrow [(\left|X\right|=\left|V\right| \vee \left|Y\right|=\left|V\right|)\rightarrow \theta(X,Y)] \notag \\
& \wedge &  \left|V\right|=\omega_1 \rightarrow [(\left|X\right|<\left|V\right| \vee \left|Y\right|<\left|V\right|)\rightarrow \left|X\right|=\left|Y\right|] \notag\\
& \wedge & \left|V\right|\neq \omega_1 \rightarrow \left|X\right|=\left|Y\right|
\end{myeqnarray}
Intuitively, this says that if the domain~$V$ is the first uncountable cardinal, then~$V$ is the only member of its own equivalence class, while all the other concepts of the same cardinality as~$V$ form an equivalence class, and then properly smaller concepts are separated into equivalence classes according to cardinality; while if the domain~$V$ is any other size, then all the concepts are separated into equivalence classes according to cardinality. Note that~$\left|V\right|=\omega_1$ can be written in second-order logic in a standard way: for, since we have global choice in the background, it simply consists in the claim that the global well-order of~$V$ has a limit point, and that any initial segment of the global well-order is bijective with the initial segment corresponding to the first limit point, but that~$V$ itself is not bijective with this initial segment. 

Let's verify that this $E$ from~(\ref{eqn:iamanequivalencerelation}) is cardinality coarsening on small concepts but not cardinality coarsening. For the former, there are two cases to consider. First suppose that the domain~$V$ has cardinality~$\omega_1$. Then by construction~$\left|\nicefrac{P(V)}{E}\right|= \omega$. Suppose that~$\left|Y\right|=\left|X\right|\leq \omega$. Then by definition~$E(X,Y)$ since~$E$ is sameness of cardinality on smaller concepts. Second suppose that the domain of~$V$ is any other cardinality. Then by construction~$E$ is just sameness of cardinality and so we are done. Now, let us note that~$E$ is not cardinality coarsening. For, consider a model where the domain~$V$ has size~$\omega_1$. Consider~$Y_0=V$ and~$X_0=Y_0\setminus\{y_0\}$, where~$y_0$ is any element of~$Y_0=V$. Then by construction we have that~$\neg E(X_0,Y_0)$, while they of course have the same cardinality.

\begin{figure}
\begin{center}
\includegraphics[width=7.5cm]{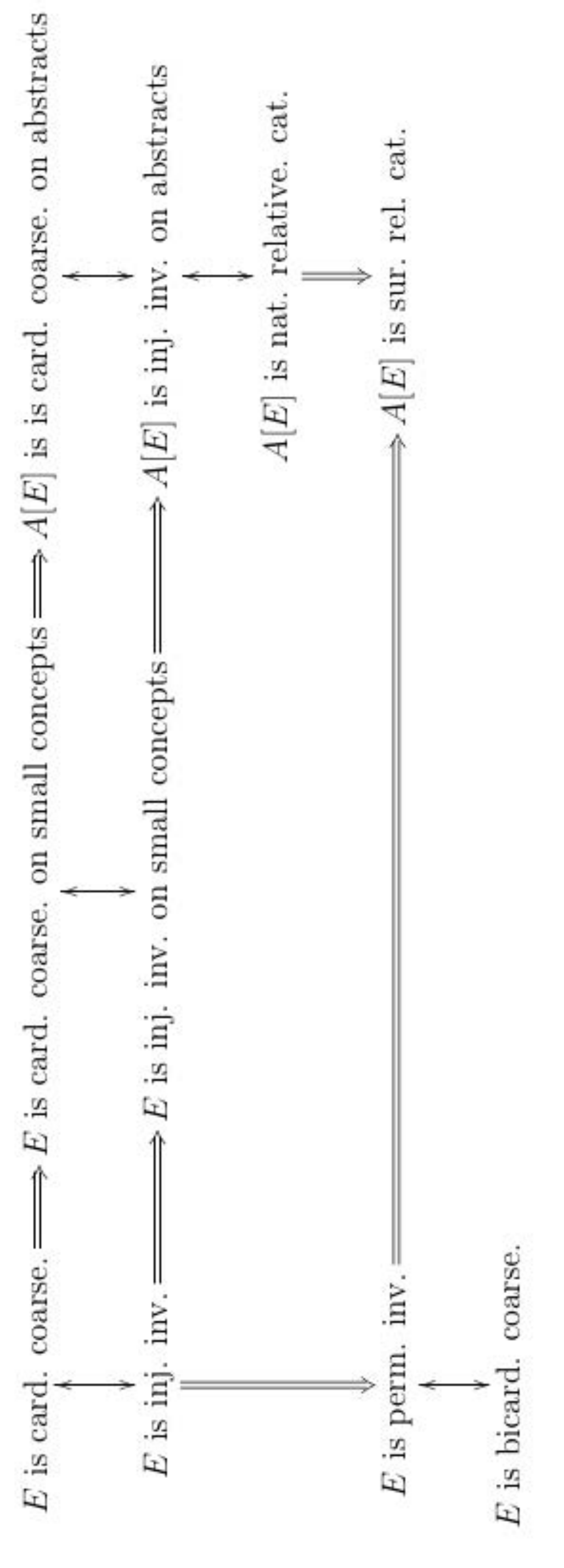}
\end{center}
\caption{Provability Relation Among Conditions on $E$ and $A[E]$}\label{diagram1}
 \end{figure}

\FloatBarrier

\section{Some Examples and Non-Examples}\label{sec04}

In this section, we give some examples and non-examples of naturally relatively categorical abstraction principles. In the case of abstraction principles which are not naturally relatively categorical, we have always been able to find a single sentence which illustrates this. So we define:
\begin{defn}\label{eqn:nateleequiv}
The abstraction principle $A[E]$ is \emph{relatively elementary equivalent} if whenever $\mathcal{M}$ is a model of $A^2[E]$ as in equation~(\ref{eqn:1}), then for any sentence $\varphi$ in the signature $L_0[\partial]$ of the induced structures $\mathcal{M}_i$ from equation~(\ref{eqn:induced}) one has that $\mathcal{M}_1\models \varphi$ if and only if $\mathcal{M}_2\models \varphi$. If $A[E]$ is not relatively elementarily equivalent, then a sentence $\varphi$ of $L_0[\partial]$ such that there is a model $\mathcal{M}$ of $A^2[E]$ with $\mathcal{M}_1\models \varphi$ and $\mathcal{M}_2\models \neg \varphi$ (or vice-versa) is called a \emph{witness} to the failure of relative elementary equivalence.
\end{defn}
\noindent Further, note that natural relative categoricity trivially implies relative elementary equivalence. Prior to stepping into the examples, let us mention that there's an obvious sense in which relative elementarily equivalence is a more apposite formalization of determinacy of truth-value ideas than natural relative categoricity itself. However, it seems difficult to study relative elementarily equivalence directly, and so we study natural relative categoricity instead, since it is the obvious sufficient condition for relative elementarily equivalence in this setting. It is not even obvious to us whether these notions are distinct in the setting of abstraction principles:
\begin{Q}\label{eqn:Q1}
Suppose that $E(X,Y)$ is an $L_0$-formula which is provably an equivalence relation in our background second-order logic and such that $A[E]$ is consistent. Suppose further that $A[E]$ is relatively elementarily equivalent. Is it necessarily the case that $A[E]$ is naturally relatively categorical? 
\end{Q}
\noindent One suspects that the answer to this question is `no,' but we have been unable to produce a counterexample.

\subsection{Hume's Principle}

Recall that Hume's Principle is the abstraction principle associated to the equivalence relation of equinumerosity. 
In \cite{Walsh2012aa} Proposition 14 p. 1687, it was shown that Hume's Principle is naturally relatively categorical. In that paper, the result was established by hand. But with our main  Theorem~\ref{thm:ncr=iia=cca} now in place, we can reduce this to the one-line observation that equinumerosity is trivially cardinality coarsening on abstracts.

\subsection{Boolos' New V}\label{sec:NewV}

In his \cite{Boolos1989aa}, Boolos drew attention to the following equivalence relation:
\begin{myequation}
E(X,Y)\equiv (\left|X\right|<\left|V\right| \vee \left|Y\right|<\left|V\right|)\rightarrow X=Y
\end{myequation}
Intuitively, this equivalence identifies all the concepts bijective with the universe~$V$, but acts like Basic Law~V on small concepts. By Theorem~\ref{thm:ncr=iia=cca}, it's easy to see that this equivalence relation is not naturally relatively categorical. For, it is trivially not cardinality coarsening on abstracts. Indeed, any two distinct small concepts which are equinumerous will not be equivalent under this equivalence relation. Further, it turns out that New~V is not relatively elementarily equivalent~(\ref{eqn:nateleequiv}), and that the witnessing sentence was previously studied by Jan\'e and Uzquiano \cite{Jane2004aa}. For, they noted that some models of New~V generated membership relations which were well-founded while others did not. 

To see this, enumerate the finite subsets of~$\omega$ as~$X_1, X_2, X_3, \ldots$, and without loss of generality, suppose that~$1\in X_1$. Then let~$\partial_1:P(\omega)\rightarrow \omega$ be any surjection that sends all the infinite subsets of~$\omega$ to~$0$, and sends~$X_1$ to~$1$. Further, let~$\partial_2:P(\omega)\rightarrow \omega$ be the surjective  map that sends all the infinite subsets of~$\omega$ to~$0$,  and acts on the finite subsets of~$\omega$ as follows:
\begin{myequation}
\partial_2(X_i) = \min(\omega\setminus (X_i \cup \{0,\partial(X_1), \ldots, \partial(X_{i\mbox{-}1})\}))
\end{myequation}
Then~$\mathcal{M}=(\omega, P(\omega), P(\omega\times \omega), \ldots, \partial_1, \partial_2)$ is model of~$A^2[E]$~(\ref{eqn:AE2}). Since both maps are surjective, we have that the induced structures~$\mathcal{M}_i$ from equation~(\ref{eqn:induced}) will satisfy $\mathcal{M}_i=(\omega, P(\omega), P(\omega\times\omega), \ldots, \partial_i)$. 

If~$A[E]$ were naturally relatively categorical, then $A[E]$ would be relatively elementarily equivalent. But let $\varphi$ be the following sentence in the signature of $L_0[\partial]$:
\begin{myequation}\label{eqn:witnsessnewv}
\varphi\equiv \exists \; X \; \exists \; b \; (\left|X\right|<\left|V\right| \; \& \; b=\partial(X) \; \& \; Xb)
\end{myequation}
Then it's easy to see that $\mathcal{M}_1\models \varphi$ since $b=\partial_1(X_1)$ is a witness. However, by construction, $\mathcal{M}_2\models \neg \varphi$. So $\varphi$ is a witness to the failure of relative elementarily equivalence of New~V, so that New~V is not naturally relatively categorical. Boolos was interested in New~V because it allowed one to define an ersatz membership relation:
\begin{myequation}\label{eqn:membership}
a\eta b \Longleftrightarrow \exists \; X \; (\left|X\right|<\left|V\right| \; \& \; \partial(X)=b \; \& \; Xa)
\end{myequation}
Expressed in these terms, $\varphi$ from (\ref{eqn:witnsessnewv}) says that there's a $b$ with $b\eta b$, which intuitively says that the membership relation from~(\ref{eqn:membership}) is not well-founded. 
\subsection{Bicardinality}\label{sec:bicardinality}

Consider the equivalence relation:
\begin{myequation}\label{eqn:bicard}
E(X,Y)\equiv \left|X\right|=\left|Y\right| \wedge \left|V\setminus X\right|=\left|V\setminus Y\right|
 \end{myequation}
It's easy to see that this is provably an equivalence relation in our background second-order logic. Further, this equivalence relation is trivially an example of an equivalence relation which is bicardinality coarsening but not cardinality coarsening, since for instance in infinite structures $V$ and $V\setminus\{a\}$ will have the same cardinality but will not be $E$-equivalent. Let's call the associated principle~$A[E]$ the \emph{Bicardinality Principle}. By Theorem~\ref{thm:coverthm}, we have that the Bicardinality Principle is surjectively relatively categorical. However, it turns out that the  Bicardinality Principle is not relatively elementarily equivalent, and hence not naturally relatively categorical. 

The non-relative elementary equivalence of the Bicardinality Principle is related to the different ways this principle can interpret the natural numbers. Just as with Hume's Principle, the Bicardinality Principle allows one to build a copy of the natural numbers $\{\overline{0}, \overline{1}, \ldots, \overline{n}, \overline{n\mbox{+}1}\,\ldots\}$ by setting $\overline{0}=\partial(\emptyset)$ and $\overline{n\mbox{+}1} = \partial(\{\overline{0}, \ldots, \overline{n}\})$. But the Bicardinality principle allows one to build a second copy $\{\widehat{0}, \widehat{1}, \ldots, \widehat{n}, \widehat{n\mbox{+}1}\,\ldots\}$ of the natural numbers by setting $\widehat{0}=\partial(V)$ and $\widehat{n\mbox{+}1} = \partial(V\setminus \{\widehat{0}, \ldots, \widehat{n}\})$. Then we can define the following sentences for each $n>0$:
\begin{myequation}
\varphi_{n} \equiv \bigwedge_{m=0}^{n-1} \overline{m}\neq \overline{n}, \hspace{10mm} \widehat{\varphi}_{n} \equiv \bigwedge_{m=0}^{n-1} \widehat{m}\neq \widehat{n},
\end{myequation}
Then one has that the Bicardinality Principle proves each $\varphi_n$ and each $\widehat{\varphi}_n$. The proof is very similar to the proof that Hume's Principle proves the variant $\psi_n$ of $\varphi_n$ wherein $\overline{n}$ is defined in terms of~$\#$ instead of~$\partial$ (where again, $\#$ is the symbol reserved for the abstraction operator featuring in Hume's Principle). Indeed, the sentences $\psi_n$ are part of the means by which one establishes the principle~Nq mentioned in the introduction \S\ref{sec01} (cf. \cite{Walsh2014aa} equation~(34) p. 111).

However, we have that $\widehat{\varphi}_n$ for each $n>0$ is a witness to the Bicardinality Principle being not relatively elementarily equivalent. Consider a model of the Bicardinality Principle of the form 
\begin{myequation}\label{eqn:anotherequation}
\mathcal{M}^{\ast} = (\omega, P(\omega), P(\omega^2), \ldots, \partial_1)
\end{myequation}
wherein~$\partial_1:P(\omega)\rightarrow \omega$ is a surjection (cf. Proposition~\ref{prop:existencesurject}). Define the injection $\iota:\omega\rightarrow \omega$ by $\iota(n)=2n$, and define $\partial_2=\iota\circ \partial_1$. Then by an argument with which we are now familiar (cf. equations~(\ref{eqn:iamausualmove0}) and (\ref{eqn:iamausualmove})), we have that the following is also a model of $A^2[E]$:
 \begin{myequation}
\mathcal{M} = (\omega, P(\omega), P(\omega^2), \ldots, \partial_1, \partial_2)
\end{myequation}
Since $\partial_1$ is surjective, we have that the induced structure~$\mathcal{M}_1$ from equation~(\ref{eqn:induced}) is identical to the structure $\mathcal{M}^{\ast}$ from~(\ref{eqn:anotherequation}), and so $\mathcal{M}_1$ models $\widehat{\varphi}_1$ since it is a model of the Bicardinality Principle. However, if $\mathcal{E}$ denotes the evens, then we have that the induced structure $\mathcal{M}_2$ is equal to the following structure 
 \begin{myequation}
\mathcal{M}_2 = (\mathcal{E}, P(\mathcal{E}), P(\mathcal{E}^2), \ldots, \partial_2\upharpoonright P(\mathcal{E}))
\end{myequation}
Then let $D=\mathcal{E}\setminus \{\partial_2(\mathcal{E})\}$. Then $D,\mathcal{E}$ have cardinality $\omega$, as do the their relative complements $\omega\setminus D, \omega\setminus \mathcal{E}$ in $\omega$. Hence since $\mathcal{M}$ models the Bicardinality Principle, we have that $\partial_2(\mathcal{E})=\partial_2(D)$. But this implies that on $\mathcal{M}_2$, we have that $\widehat{0}=\widehat{1}$, so that $\mathcal{M}_2$ does not model $\widehat{\varphi}_1$. The same argument works for $\widehat{\varphi}_n$ for all $n\geq 1$.

\subsection{The Nuisance Principle}\label{sec:nuisances}

Consider the equivalence relation $E(X,Y) \equiv \left|X\triangle Y\right| < \omega$ wherein $X\triangle Y$ denotes the symmetric difference $X\triangle Y = (X\setminus Y)\cup (Y\setminus X)$. In this, recall that the notion of $\left|Z\right|<\omega$ is that of Dedekind-finiteness~(\ref{eqn:cardinalityabb:3}), so that $E$ is $L_0$-definable. But per Proposition~\ref{eqn:equivchardedfin}, this aligns with the notion of being finite in the sense of being bijective with the initial segments $I_a=\{x: x<a\}$ from our global well-order where $a$ is finite, i.e. is less than the first limit point. So while $E$ is expressible purely in terms of the signature of the background second-order logic $L_0$ (cf. Definition~\ref{defn:L000}), we can use the global well-order to show that it is an equivalence relation. For, using the global well-order we can show that the union of any two finite sets is finite, and so the transitivity of $E$ follows from $
X\triangle Z \subseteq (X\triangle Y)\cup (Y\triangle Z)$. 

The abstraction principle $A[E]$ has been called the ``Nuisance Principle'' following Wright's identification of its abstracts as ``nuisances''   \cite{Wright1997aa}, and is related to an earlier principle presented by Boolos  \cite{Boolos1990}. Let's show that $A[E]$ is naturally relatively categorical, and let's begin with the following proposition. This proposition is well-known if one assumes the standard semantics, but we know of no extant proof for the Henkin semantics with  choice principles. In the statement of this theorem, the notion of finiteness is Dedekind-finiteness~(\ref{eqn:cardinalityabb:3}) or the aforementioned equivalent characterization (cf. again Proposition \ref{eqn:equivchardedfin}).
\begin{prop}\label{prop:nufinite}
$A[E]$ implies that $V$ is finite.
\end{prop}
\begin{proof}
So suppose not. Then let's work deductively in the theory $A[E]$ under the assumption that the universe $V$ is infinite. Let $\omega$ denote the least limit point in the global well-order. First let's establish that there's a way to enumerate all finite concepts:
\begin{claim}\label{nu:claim0}
There is a ternary relation $R$ such that for all $X$, $\left|X\right|<\omega$ if and only if there is $a<\omega$ and $b$ such that $R[a,b]=X$, where again $R[a,b]=\{c: R(a,b,c)\}$.
\end{claim}
\noindent This claim follows trivially by comprehension from the following claim:
\begin{claim}\label{nu:claim1}
There is a ternary relation $R$ such that for all $a<\omega$, and all $X$ with $\left|X\right|\leq \left|I_a\right|$ there is $b$ with $R[a,b]=X$, where again $R[a,b]=\{c: R(a,b,c)\}$.
\end{claim}
\noindent This claim of course follows from the following claim by a single application of ${\tt AC}$~(\ref{eqn:AC}):
\begin{claim}\label{nu:claim2}
For all $a<\omega$ there is binary relation $R_a$ such that for all $X$ with $\left|X\right|\leq \left|I_a\right|$ there is $b$ with $R_a[b]=X$.
\end{claim}
\noindent We argue for this latter claim by induction on $a$. For $a=0$, it holds trivially since we may choose $R_a$ equal to the empty binary concept. Now suppose it holds for $a_0$ with witness $R_{a_0}$, and suppose $a=s(a_0)$, where this is the partial successor operation from~(\ref{eqn:defnsucconord}). By Infinite Products are Maxs~(\ref{eqn:InfProdMax}), there is a bijection $\langle \cdot, \cdot\rangle: V\times V\rightarrow V$. Then we define the binary relation $R_{a}$ as follows:
\begin{myequation}
R_{a}[b]=\{z: \exists \; b_0,c_0 \; b=\langle b_0,c_0\rangle \; \& \; (z=c_0 \vee R_{a_0}[b_0](z)) \}
\end{myequation}
So this finishes the proof of Claim~\ref{nu:claim2} and with it the proof of Claim~\ref{nu:claim0}. 

Using this enumeration of the finite concepts, let's define a partial injective map from finite concepts to objects. For this, let us fix a ternary relation $R$ as in Claim~\ref{nu:claim0}. Then consider the following partial map from concepts to objects:
\begin{myequation}
\nabla(X)=c \Longleftrightarrow \left|X\right|<\omega \; \& \; c=\min_{<}\{\langle a,b\rangle: a<\omega \; \& \; R[a,b]=X\}
\end{myequation}
Then this map is defined on all finite concepts, and on these, it is an injection. For, supposing that $\nabla(X)=\langle a,b\rangle=\nabla(Y)$, we have $X=R[a,b]=Y$. Hence, indeed $\nabla$ is a partial map from concepts to objects which is defined and injective on finite concepts. 

So now consider $\mathrm{rng}(\partial)$, where $\partial$ as usual is the abstraction operator associated to $A[E]$. As in the argument for (\ref{eqn:niceappchoice}) in the previous section, we may use ${\tt AC}$~(\ref{eqn:AC}) to show that there is a binary relation $S$ such that for all $x$ from $\mathrm{rng}(\partial)$ one has that $\partial(S[x])=x$. Then for all concepts $X$, there is unique $x$ from $\mathrm{rng}(\partial)$ such that $E(X,S[x])$. Hence there are finite concepts $F_X, G_X$ such that $X=(S[x]\cup F_X)\setminus G_X$. Then define a map $\partial^{\prime}$ from concepts to objects by $\partial^{\prime}(X)=\langle x,\langle a,b\rangle \rangle$ iff $x$ is from $\mathrm{rng}(\partial)$ and $E(X,S[x])$ and 
\begin{myequation}\label{eqn:nurussell}
\langle a,b\rangle = \min_<\{\langle c,d\rangle: \exists \; F_X, G_X \; X=(S[x]\cup F_X)\setminus G_X \; \& \; \nabla(F_X)=c \; \& \; \nabla(G_X)=d\}
\end{myequation}
Then by construction $\partial^{\prime}$ is an injection from concepts to objects, and so using the full comprehension schema, we may again replicate the Russell paradox to derive a contradiction.
\end{proof}
\noindent Using this result, we can now deduce:
\begin{prop}
$A[E]$ is naturally relatively categorical.
\end{prop}
\begin{proof}
Let's work within $A[E]$. By the above result, $V$ is finite. Hence, since $\left|X \triangle Y \right|<\omega$ for all finite $X$ and $Y$, we have that $E(X,Y)$ for all $X,Y$. And then trivially $A[E]$ is cardinality coarsening on abstracts. Hence by Theorem~\ref{thm:ncr=iia=cca}, $A[E]$ is naturally relatively categorical.
\end{proof}\noindent So this proof shows us that naturally relatively categorical abstraction principles are not pairwise consistent. For, as noted earlier, Hume's Principle is naturally relatively categorical and yet it implies that $V$ is infinite, since for instance when one defines the natural numbers using the resources of Hume's Principle, they will be Dedekind-infinite. But of course one would not have expected relatively categorical concepts to be pairwise consistent. For instance, the standard relative categoricity argument for set theory can be deployed to show that the standard axioms plus ``there are no inaccessibles'', as well as the standard axioms plus ``there are inaccessibles,'' are both relatively categorical, at least assuming (as we are here) that the domain of the model is the entire universe. In \S\ref{sec05stable}, we'll note by contrast that equivalence relations which are cardinality coarsening on small concepts are pairwise consistent. So this is another way of seeing that the conditions on $A[E]$ in Figure~\ref{diagram1} don't have implications for the analogous conditions on $E$: while the Nuisance Principle is naturally relatively categorical and hence cardinality coarsening on abstracts, it is not the case that the underlying equivalence relation is cardinality coarsening on small concepts.

\subsection{The Complementation Principle}

Let's say that a concept $X$ is a \emph{complement}, abbreviated $\mathrm{Cmp}(X)$ if there is $Y\approx X$ such that~$X\sqcup Y = V$. Then we define:
\begin{myequation}
E(X,Y)\equiv (\mathrm{Cmp}(X) \vee \mathrm{Cmp}(Y))\rightarrow (X=Y \vee (X\approx Y \; \& \; X\sqcup Y =V))
\end{myequation}
\noindent This equivalence relation sorts out concepts as follows: each pair of equinumerous complementary sets constitute a separate equivalence class, everything else occupies a single ``junk'' equivalence class. 

This abstraction principle has finite models only for odd domains, and domains of sizes 2 and 4. For, in the case of odd domains, there are no complements and so there is only one equivalence class. In the case of even domains of size $2n$, the number of $E$-equivalence classes is exactly $f(n)=\frac{1}{2}\binom{2n}{n}+1$, and one has that $f(1)=2, f(2)= 4$ and $f(n)\geq 2n+1$ for $n\geq 3$, as one can establish by induction on $n$. 

Unlike the Nuisance Principle, this $A[E]$ is not naturally relatively categorical and indeed not relatively elementary equivalent. For, consider the domain $M=\{a,b,c,d\}$ of size $4$, where there are exactly 4 equivalence classes. Define two abstraction operators by \begin{myequation}\label{eqn:comp-counter}
\begin{array}{cccc}
\partial_1(\{a,b\}) = a, & \partial_1(\{a,c\})= b, & \partial_1(\{a,d\})= c, & \partial_1(\emptyset) = d, \\
\partial_2(\{a,b\}) = c, & \partial_2(\{a,c\})= b, & \partial_2(\{a,d\})= a, & \partial_2(\emptyset) = d 
\end{array}
\end{myequation}
This determines a model $\mathcal{M}$ of $A^2[E]$ with first-order part $M$. Then consider the sentence: 
\begin{myequation}\label{eqn:CP-witness}
\forall X\; [|X|=2 \rightarrow \exists Y (E(X,Y) \wedge Y(\partial_i(X))\wedge Y(\partial_i(\emptyset))]
\end{myequation}
This sentence fails in $\mathcal{M}_1$ with witness $X=\{a,b\}$ but holds in $\mathcal{M}_2$. Hence $A[E]$ is not relatively elementary equivalent and hence not naturally relatively categorical. But like with the Nuisance Principle, $A[E]$ 
proves $V$ is finite.  
Hence, it is not this feature alone which permitted the Nuisance Principle to be naturally relatively categorical.


\section{Stability Criteria and the Bad Company Problem}\label{sec05stable}

One of locus of activity on abstraction principles in recent years has been stability criteria (cf. Cook~\cite{Cook2012aa}). Some of the key stepping stones in this hierarchy are the following:

\begin{myenumerate}

\item The abstraction principle $A[E]$ is \emph{stable} if there is some cardinal~$\lambda$ such that, for all~$\kappa\geq \lambda$,~$A[E]$ has a standard model of size~$\kappa$.\label{eqn:defn:stable}

\item The abstraction principle $A[E]$ is \emph{strongly stable} if there is some cardinal~$\lambda$ such that, for all cardinals~$\kappa$,~$A[E]$ has a standard model of size~$\kappa$ if and only if~$\kappa\geq\lambda$.\label{eqn:defn:superstable}

\end{myenumerate}
\noindent As above, a structure~$\mathcal{M}=(M, S_1[M], S_2[M], \ldots)$ is said to be  \emph{standard} if~$S_n[M]=P(M^n)$ for each~$n\geq 1$ (cf. \S~\ref{sec02}). Further, here one identifies the \emph{cardinality of~$\mathcal{M}$} as~$\left|M\right|$, i.e. the cardinality of its first-order part. So this convention on the cardinality of the many-sorted structure $\mathcal{M}$ differs from other settings in which it would be more natural to define its cardinality as the cardinality of the union of its various domains  (cf. \cite{Walsh2014aa} \S{6.1} p. 107 equation (18), \cite{Manzano1996} p. 231 and \cite{Ebbinghaus1985aa} pp. 32, 64).

These criteria have been proposed as solutions to the so-called \emph{Bad Company Problem}, namely the problem of identifying criteria on abstraction principles which would ensure consistency of $A[E]$; perhaps also joint consistency with like abstraction principles and the ability to interpret mathematics of various sorts. Of course, the relative categoricity notions studied here are not candidates for a solution to the Bad Company Problem since they do not ensure consistency. However, the sufficient conditions on the equivalence relations themselves that we have identified for relative categoricity \emph{do} ensure consistency, and indeed stability and joint consistency. In particular, we have the following results:
  
\begin{prop}\label{prop:stabileforcc} Suppose that $E(X,Y) $ is an $L_0$-formula and suppose that~$E$ is cardinality coarsening on small concepts. Then for all infinite cardinals~$\kappa$ there is a standard model of~$A[E]$ of cardinality~$\kappa$.
\end{prop}

\begin{prop}\label{prop:stableforbcc} Suppose that $E(X,Y) $ is an $L_0$-formula and suppose that $E$ is bicardinality coarsening. Then for all infinite cardinals $\kappa$ there is a standard model of $A[E]$ of cardinality $\kappa$.
\end{prop}

These two propositions imply that the equivalence relations satisfying these constraints are stable~(\ref{eqn:defn:stable}). However, this cannot be improved to strong stability~(\ref{eqn:defn:superstable}). For, consider the following equivalence relation:
\begin{myequation}\label{eqn:SEDcounter}
E(X,Y) \equiv[\left|V\right|<\omega \; \& \; \textrm{$|V|$ is even}\rightarrow X\approx Y] \wedge [\neg(\left|V\right|<\omega \; \& \; \textrm{$|V|$ is even})\rightarrow X=X]
\end{myequation}
This equivalence relation is cardinality coarsening on small concepts. In the case where $\left|V\right|<\omega$ is even, we have that $E$ is equinumerosity, so that $\left|\nicefrac{P(V)}{E}\right|\geq \left|V\right|+1$, all concepts are small, and thus trivially $X\approx Y $ implies $E(X,Y)$. But in the other case, $E$ says that all concepts are equivalent, so that again we trivially have that $X\approx Y $ implies $E(X,Y)$. So this $E$ is cardinality coarsening on small concepts, but $A[E]$ is not strongly  stable~(\ref{eqn:defn:superstable}). For $A[E]$ is satisfiable on all finite  odd domains but no finite even domains.



Cook~\cite{Cook2012aa} recommends strong stability as the appropriate solution to the Bad Company Problem, and so the previous examples indicate that the solution offered by cardinality coarsening on small concepts and bicardinality coarsening will be extensionally distinct from Cook's solution. However, it is worth noting that  
Propositions~\ref{prop:stabileforcc} and~\ref{prop:stableforbcc} imply that if $E$ is cardinality coarsening on small concepts (respectively, bicardinality coarsening), all structures witnessing that $E$ is not strongly stable will be  finite.  



But cardinality coarsening on small concepts and bicardinality coarsening have a more mixed scorecard on  other desiderata  taken to be relevant to judging the success of  proposed solutions to the Bad Company problem. On the one hand, Proposition~\ref{prop:stabileforcc} and Proposition~\ref{prop:stableforbcc} imply that all these abstraction principles are jointly consistent since they all have models of any infinite cardinality. But by the same token, our conditions on $E$ restrict the mathematics that can be recovered in the theory $A[E]$.  Inspection of the below proof of these 
propositions show that the joint theory of all these abstraction principles will be interpretable in second-order Peano arithmetic ${\tt PA}^2$ (cf. \cite{Simpson1999} pp. 4 ff for the definition of ${\tt PA}^2$ and cf. \cite{Walsh2014aa} \S{2} for the definition of interpretation). This is of course because a standard model of cardinality~$\omega$ is the underlying first-order domain of the standard model of ${\tt PA}^2$, and the proofs of the propositions for the case $\kappa=\omega$ are formalizable in ${\tt PA}^2$. Hence, there will be no hope of interpreting, e.g., standard ${\tt ZFC}$-set theory by restricting attention to such abstraction principles, because ${\tt ZFC}$ proves the consistency of ${\tt PA}^2$ and so ${\tt PA}^2$ cannot interpret ${\tt ZFC}$.

But it's difficult to  compare this directly to the situation of stability or strong stability. For, suppose that $\Phi$ is a sentence of second-order logic that is true on all and only standard models of cardinality $\geq \lambda$. Then $A[E]$ will be strongly stable for $E(X,Y)\equiv (X=Y) \vee \Phi$ (disjunctive abstraction principles like this are due to Heck~\cite{Heck-Jr.1992aa}). So such $A[E]$ will be able to encode a good deal of interpretability strength, at least when interpretability with parameters is allowed (cf. \cite{Walsh2014aa} \S{2} for the definition of interpretations with parameters). For instance, $\Phi$ could say that there is a model of ${\tt ZFC}^2$. It's unknown to us whether the interpretability strength-- where the interpretation is \emph{without} parameters-- of the theories associated to all stable or strongly stable abstraction principles is strictly above ${\tt PA}^2$. This might be related to the fact that stability conditions often depend on undecidable questions in set theory.  For instance, the claim that  the abstraction principle New V is strongly stable is equivalent to the generalized continuum hypothesis (cf. \cite{Shapiro1999ac} p. 315,  \cite{Cook2007aa} p. 297). Our conditions, by contrast, concern only what is demonstrable in a sound and complete deductive system.
	
So let's proceed to the proofs of Proposition~\ref{prop:stabileforcc} and Proposition~\ref{prop:stableforbcc}. In these proofs, we employ the notion on equivalence classes from the close of \S\ref{sec02}. We begin with a helpful lemma: 

\begin{lem}\label{lem:dichtomylem}
(\emph{Coarsening Dichotomy}) Suppose that~$E$ is cardinality coarsening on small concepts. Then for all infinite~$\kappa$ we have  $\left| \nicefrac{P(\kappa)}{E}\right| <\kappa$ or $\left| \nicefrac{P(\kappa)}{E}\right|=\left| \nicefrac{P(\kappa)}{\approx}\right|=\kappa$, 
and hence in either case we have~$\left|\nicefrac{P(\kappa)}{E}\right|\leq \kappa$.
\end{lem}
\begin{proof}
So suppose that~$\left| \nicefrac{P(\kappa)}{E}\right| \geq \kappa$. We must show that~$ \left| \nicefrac{P(\kappa)}{E}\right|=\left| \nicefrac{P(\kappa)}{\approx}\right|=\kappa$. Since~$\left|\nicefrac{P(\kappa)}{E}\right| \geq \kappa$, we of course have that $X\subseteq \kappa$ implies $\left|X\right|\leq \left|\nicefrac{P(\kappa)}{E}\right|$. 
Then we have:
\begin{myequation}
[X,Y\subseteq \kappa \; \& \; \left|Y\right|= \left|X\right|] \Longrightarrow [\left|Y\right|=\left|X\right|\leq \left|\nicefrac{P(\kappa)}{E}\right|] 
\end{myequation}
and the consequent of this conditional clearly implies  $E(X,Y)$ since~$E$ is cardinality coarsening on small concepts. In terms of the equivalence classes, this says that the partition of~$P(\kappa)$ by~$E$ is coarser than the partition of~$P(\kappa)$ by equinumerosity. Then choose representatives~$X_0, X_1, \ldots$ for the~$E$-equivalence classes, so that~$\nicefrac{P(\kappa)}{E}=\bigsqcup_{i\in I} [X_i]_{E}$. Then define~$I:\nicefrac{P(\kappa)}{E}\rightarrow \nicefrac{P(\kappa)}{\approx}$ by~$I([X_i]_E)=[X_i]_{\approx}$. Then we claim that~$I$ is an injection. For, suppose that~$I([X_i]_E)=I([X_j]_E)$ so that~$[X_i]_{\approx}=[X_j]_{\approx}$. Then~$\left|X_i\right|=\left|X_j\right|$ and so by the previous equation we have~$E(X_i,X_j)$ and hence~$[X_i]_E=[X_j]_E$ and so~$i=j$ since the~$X_i,X_j$ are representatives of the~$E$-equivalence classes. Since~$I:\nicefrac{P(\kappa)}{E}\rightarrow \nicefrac{P(\kappa)}{\approx}$ is indeed injective, we then have the following, where the last inequality appeals to the infinitude of~$\kappa$: $\left| \nicefrac{P(\kappa)}{E}\right| \leq \left| \nicefrac{P(\kappa)}{\approx}\right| \leq \kappa$. Since we're assuming~$\kappa\leq \left| \nicefrac{P(\kappa)}{E}\right|$, we are done.
\end{proof}
So we can now prove Proposition~\ref{prop:stabileforcc}:
\begin{proof}
Let~$\kappa$ be an infinite cardinal, and consider the standard structure, where $<$ is the natural well-ordering on the cardinal: $
\mathcal{M}_0=(\kappa, P(\kappa), P(\kappa\times \kappa), \ldots, <)$. 
By the Coarsening Dichotomy, we have that~$\left|\nicefrac{P(\kappa)}{E}\right| \leq \kappa$. Choose an injection~$\widehat{\partial}: \nicefrac{P(\kappa)}{E}\rightarrow  \kappa$. Then we may define the map~$\partial:P(\kappa)\rightarrow \kappa$ by~$\partial(X)=\widehat{\partial}([X]_E)$. Then by construction we have that~$\mathcal{M}=(\kappa, P(\kappa), P(\kappa\times \kappa), \ldots, <, \partial)$ is a model of~$A[E]$ of cardinality~$\kappa$.
\end{proof}

Similarly, we can prove Proposition~\ref{prop:stableforbcc}:
\begin{proof}
Let's use $E_0$ as an abbreviation for the equivalence relation of bicardinality.  So since $E$ is bicardinality coarsening, by definition we have that $
E_0(X,Y)$ implies $E(X,Y)$. 
So as in the proof of the Coarsening Dichotomy, this implies that $\left|\nicefrac{P(\kappa)}{E}\right|\leq \left|\nicefrac{P(\kappa)}{E_0}\right|$. Then one may argue that $
\left|\nicefrac{P(\kappa)}{E}\right|\leq \left|\nicefrac{P(\kappa)}{E_0}\right|  \leq \kappa \cdot \kappa \leq \kappa$. 
 Since~$\kappa$ is infinite, the only non-trivial inequality is the second. For this, consider the map $[X]_{E_0} \mapsto \langle \left|X\right|, \left|\kappa\setminus X\right|\rangle$.  This is an injection because we have by definition of~$E_0$ that $E_0(X,Y)$ iff $ \langle \left|X\right|, \left|\kappa\setminus X\right|\rangle=\langle \left|Y\right|, \left|\kappa\setminus Y\right|\rangle$. From $\left|\nicefrac{P(\kappa)}{E}\right|\leq \kappa$, we can use the same construction as employed in the proof of the above proposition to build models of $A[E]$.
\end{proof}

Since $E$ being bicardinality coarsening is a sufficient condition for $A[E]$ being surjectively relatively categorical, it's natural to ask about when there are standard models of $A[E]$ on which the abstraction operator is surjective.
\begin{prop}\label{prop:existencesurject}
Suppose that $E$ is bicardinality coarsening. Then for all infinite cardinals $\kappa$, there is a standard model of $A[E]$ of cardinality $\kappa$ wherein the abstraction operator is surjective if and only if $\left|\nicefrac{P(\kappa)}{E}\right|=\left|\{\lambda: \lambda<\kappa\}\right| = \kappa$.
\end{prop}
\begin{proof}
First suppose that there is a standard model of $A[E]$ of cardinality $\kappa$ where the abstraction operator is surjective. Then surjectivity implies that $\kappa= \left|\nicefrac{P(\kappa)}{E}\right|$. Then as in the proof of the Proposition~\ref{prop:stabileforcc}, we may argue that $\kappa= \left|\nicefrac{P(\kappa)}{E}\right|\leq  \left|\nicefrac{P(\kappa)}{E_0}\right| \leq \kappa$, so that in fact we have an equality. Since $\kappa$ is infinite, one further has that $\left|\nicefrac{P(\kappa)}{E_0}\right|=\left|\nicefrac{P(\kappa)}{\approx}\right|$. But trivially one has that $\left|\nicefrac{P(\kappa)}{\approx}\right|=\left|\{\lambda: \lambda<\kappa\}\right|$. For the converse, suppose that $\left|\nicefrac{P(\kappa)}{E}\right|=\left|\{\lambda: \lambda<\kappa\}\right| = \kappa$, so that trivially $\left|\nicefrac{P(\kappa)}{E}\right|= \kappa$. Then the model construction from the previous proposition yields a model of $A[E]$ in which the abstraction operator is surjective. 
\end{proof}

\section{Antonelli and Fine on Logicality and Invariance}\label{sec05:aldo}

The famous Tarski-Sher thesis on logicality (\cite{Tarski1986aa}, \cite{Sher1991aa}) suggests that logical notions are those that are invariant under all permutations of the domain. This can be made precise as follows. If~$\pi:M\rightarrow M$ is a permutation, then~$\pi$ induces permutations~$\overline{\pi}:M^n\rightarrow M^n$ and~$\overline{\pi}:P(M)\rightarrow P(M)$. By iterating, a permutation~$\pi: M\rightarrow M$ induces a permutation of any sort of~$\omega$-th order higher-order structure with first-order domain~$M$. Then one says:
\begin{myenumerate}
\item A subset~$X$ of any sort of this structure is \emph{permutation-invariant} if~$\overline{\pi}(X)=X$ holds for all permutations~$\pi:M\rightarrow M$.\label{eqn:defn:permTS}
\end{myenumerate}
So part of what we want to do in this section is to carefully distinguish this notion coming from Tarski-Sher from our notion of permutation invariance for abstraction principles~(\ref{eqn:defn:piI}). To illustrate the Tarski-Sher notion, consider~$X=\{(A,B,C)\in (P(M))^3: A\cap B=C\}$. Then for any permutation~$\pi:M\rightarrow M$, we have 
\begin{myequation}
A\cap B=C \Longleftrightarrow \overline{\pi}(A\cap B)=\overline{\pi}(C) \Longleftrightarrow  \overline{\pi}(A)\cap \overline{\pi}(B)=\overline{\pi}(C) 
\end{myequation}
So intersection comes out logical in this sense, and similarly for all the other boolean operations. 

Antonelli (\cite{Antonelli2010aa}) distinguishes several ways that abstraction principles and their associated equivalence relations can be invariant. One way is in terms of the abstraction operator itself. Since the abstraction operator is a function from concepts to objects, its graph is a certain collection of ordered pairs of concepts and objects. Hence an abstraction principle will be logical in the sense of Tarski-Sher just if this collection of ordered pairs is closed under the taking of permutations. This of course happens if and only if $\partial(X)=a$ implies $\partial(\overline{\pi}(X))=\pi(a)$ and vice-versa for all permutations $\pi:V\rightarrow V$.  And this is equivalent to the requirement that  $\partial(\overline{\pi}(X))=\pi(\partial(X))$ for all permutations $\pi:V\rightarrow V$. Antonelli calls abstraction principles which satisfy this condition \emph{objectually invariant} (\cite{Antonelli2010aa} p. 286). However, objectually invariant abstraction operators are rare. For instance, Hume's Principle is not objectually invariant, nor is the Bicardinality Principle. Hence, if one insisted upon the logicality of abstraction \emph{operators} in the sense of Tarski-Sher, then as Antonelli notes this would rule out most interesting examples (cf. \cite{Antonelli2010aa} p. 286).

The notion which we are calling permutation invariance~(\ref{eqn:defn:piI}) is called \emph{simple invariance} by Antonelli (\cite{Antonelli2010aa} p. 286). This notion was studied earlier in Burgess' discussion of Fine's work on abstraction principles (\cite{Burgess2005} p. 171). However, the notion which Fine himself was primarily interested in was the permutation invariance of the equivalence relation itself (cf. \cite{Fine2002} p. 111):
\begin{myequation}\label{eqn:finevinara}
\forall \; \mbox{ bijection } \pi:V\hspace{-1mm}\rightarrow\hspace{-1mm}V \; (E(X,Y)\rightarrow E(\pi(\overline{X}), \pi(\overline{Y}))
\end{myequation}
This of course is just the logicality of the equivalence relation~$E$ in the sense of Tarski-Sher. However, Fine rather motivates this constraint on equivalence relations by reference to Frege's ideas about the generality of logic (\cite{Fine2002} p. 109, cf. \cite{MacFarlane2002aa} p. 34), and does not at all invoke the Tarski-Sher thesis. But whatever its motivation, all the equivalence relations $E$ studied here will satisfy this constraint since they are assumed to be formulas in the background signature $L_0$ (cf. Definition~\ref{defn:L000}), and so will all be invariant under permutations in the sense of~(\ref{eqn:finevinara}).

Antonelli's work on abstraction principles and notions of permutation invariance (\cite{Antonelli2010aa}) complements his work on abstraction principles and generalized quantifiers (\cite{Antonelli2010ac,Antonelli2010ab}). In this work, the idea was to think about the equivalence relations featuring in abstraction principles as examples of generalized quantifiers. For instance, the truth-condition of the sentence ``Just as many students [A] as teachers [B] are hockey fans [C],'' is given by the generalized quantifier $Q(A,B,C)\equiv \left|A\cap C\right|=\left|B\cap C\right|$. So one could then view  Hume's Principle as yielding a way to provide generalized quantifiers with first-order truth-conditions. The connection to permutation invariance in the sense of (\ref{eqn:defn:permTS}) and (\ref{eqn:finevinara}) is that this is one condition among many that have been developed for assaying which generalized quantifiers actually occur in natural language (cf. \cite{Peters2008aa} pp. 157, 330,  \cite{Keenan1985aa} p. 77).

However, neither Antonelli nor Burgess is committed to there being any positive reason for insisting on the constraint of permutation invariance in the sense of~(\ref{eqn:defn:piI}). It clearly does not follow from the Tarski-Sher logicality constraint on the abstraction operator or on the underlying equivalence relation. Further, it's not obviously required by the generality of logic since there are many equivalence relations which are expressible purely in the language~$L_0$ of our background second-order logic which are not permutation invariant in the sense of~(\ref{eqn:defn:piI}). However, our Theorem~\ref{thm:coverthm} suggests an instrumental reason to be interested in equivalence relations which are permutation invariant in the sense of~(\ref{eqn:defn:piI}). For, it indicates that its a necessary and sufficient condition for a type of relative categoricity. So if one was interested in determinacy of truth-value in the sense of natural relative categoricity, this result would give one a reason to be interested in permutation invariance in the sense of~(\ref{eqn:defn:piI}).

Fine was himself interested in determinacy and established several categoricity results. His first result concerns what we called cardinality coarsening~(\ref{eqn:RC3xsemi}), a notion which Fine calls \emph{numericality} (\cite{Fine2002} p. 126). This first categoricity theorem of Fine's can be stated in our terminology as follows:
\begin{thm}\label{thm:finethm} (\cite{Fine2002} p. 126) Suppose that $\mathcal{M}$ is a model of $A^2[E]$ as in equation~(\ref{eqn:1}) where $E$ is cardinality coarsening, and suppose that $\mathcal{M}$ models $\left|M\setminus \mathrm{rng}(\partial_1)\right|=\left|M\setminus \mathrm{rng}(\partial_2)\right|$. For each $i\in \{1,2\}$, consider the following induced models
 \begin{myequation}\label{eqn:fineinduced}
\mathcal{N}_i=(M, S_1[M], S_2[M], \ldots, \partial_i)
\end{myequation}
 Then $\mathcal{N}_1$ is isomorphic to $\mathcal{N}_2$.
\end{thm}
\noindent This theorem of Fine's occurs far within his book \cite{Fine2002} which treats a great variety of topics and issues. In our view, this result deserves to be better known-- for instance, it is not discussed in the \emph{Philosophical Studies} book symposium on Fine's book (cf. \cite{Fine2005ac}), nor is it treated in Burgess' discussion of Fine (\cite{Burgess2005} Chapter~3). For this reason and for the sake of completeness, let us record its proof:
\begin{proof}
By the hypothesis, there is a bijection from  $M\setminus \mathrm{rng}(\partial_1)$ to $M\setminus \mathrm{rng}(\partial_2)$. But the natural bijection $\Gamma$ is a bijection from $\mathrm{rng}(\partial_1)$ to $\mathrm{rng}(\partial_2)$. By joining these two bijections, we may obtain a  bijection~$\Delta:M\rightarrow M$ such that $\Delta \upharpoonright \mathrm{rng}(\partial_1)$ is the natural bijection~$\Gamma:\mathrm{rng}(\partial_1)\rightarrow \mathrm{rng}(\partial_2)$. Extend to~$\overline{\Delta}:\mathcal{N}_1\rightarrow \mathcal{N}_2$ by setting~$\overline{\Delta}(X)=\{\Delta(x): x\in X\}$. Then we claim that~$\overline{\Delta}:\mathcal{N}_1\rightarrow \mathcal{N}_2$ is an isomorphism. Let~$X\in S_1[M]$. We must show that $\overline{\Delta}(\partial_1(X)) = \partial_2(\overline{\Delta}(X))$. 
Since~$\Delta$ extends the natural bijection, this is equivalent to  $ \partial_2(X) = \partial_2(\overline{\Delta}(X))$. 
And since~$\mathcal{M}\models A^2[E]$, this is equivalent to~$E(X, \overline{\Delta}(X))$. But since~$E$ is cardinality coarsening, of course $E(X, \overline{\Delta}(X))$ follows from~$\overline{\Delta}(X)\approx X$.
\end{proof}

Some special cases of Fine's Theorem can be viewed as a combination of certain parts of our results. For instance, if $E$ is cardinality coarsening, then $E$ is bicardinality coarsening (cf. Figure~\ref{diagram1}) and hence $A[E]$ is surjectively relatively categorical by one direction of Theorem~\ref{thm:coverthm}. This same result follows from the special case of Fine's theorem wherein one assumes that $M\setminus \mathrm{rng}(\partial_1)$ and $M\setminus \mathrm{rng}(\partial_2)$ are both empty. However, nothing in Fine's Theorem concerns the other direction of our results, namely that certain conditions like cardinality coarsening on small concepts and bicardinality coarsening are necessary for various forms of relative categoricity. Another difference between our work and Fine's work is with the precise induced models which figure in the statement of the results. In particular, the induced structures~$\mathcal{N}_i$ in (\ref{eqn:fineinduced}) are different from the induced structures~$\mathcal{M}_i$ from (\ref{eqn:induced}) with which we have been working  in that the structures~$\mathcal{M}_i$ restrict their first-order domain down to the range of the abstraction operator~$\partial_i$. Of course, these two types of induced models will align in the case where the abstraction operators are surjective. So while Fine's theorem predicts that Hume's Principle is, in our terminology, surjectively relatively categorical, it does not obviously have any implications for the natural relative categoricity of Hume's Principle.

Fine's second categoricity theorem concerns a general theory of abstraction which he terms~${\tt GA}^{+}$. Roughly, this theory is the amalgamation of abstraction principles $A[E]$ which satisfy a combination of Tarski-Sher permutation invariance conditions~(\ref{eqn:defn:permTS}) and varieties of stability-like conditions as discussed in \S\ref{sec05stable}. For the sake of simplicity, let's consider the case of a single such abstraction principle~$A[E]$. Fine's second theorem then indicates that any model $\mathcal{M}$ of $A^2[E]$ which satisfies the constraint that $M\setminus \mathrm{rng}(\partial_1)$ and $M\setminus \mathrm{rng}(\partial_2)$ have the same cardinality will be such that the induced structures $\mathcal{N}_i$ in (\ref{eqn:fineinduced}) are isomorphic. The official statement of Fine's second theorem is more complicated since one must precisely define the analogue $A^2[E_1, E_2, \ldots]$ of $A^2[E]$ in the case where abstraction principles associated to $E_1, E_2, \ldots$ are present. See Fine~\cite{Fine2002} p.~189 for the precise statement of this result, and see Fine~\cite{Fine2002} p.~170 for the precise statement of the theory~${\tt GA}^{+}$. 

In the subsequent section of his book (\cite{Fine2002} pp.~189~ff), Fine proceeds to examine the interpretability strength of~${\tt GA}^{+}$ and related theories. It is in the context of surveying these results (cf. \cite{Burgess2005} p. 171) that Burgess introduces the notion what we have called permutation invariance~(\ref{eqn:defn:piI}), which does not occur in Fine's own work. In particular, Burgess introduces a theory (cf. \cite{Burgess2005} p. 173) in which one assumes that the equivalence relations~$E$ are permutation invariant in our sense~(\ref{eqn:defn:piI}) and further that there are at most two equivalence classes:
\begin{myequation}
\exists \; X,Y \; \forall \; Z \; (E(Z,X) \vee E(Z,Y))
\end{myequation}
This last condition is less restrictive than one might initially expect, since Burgess is working in a setting where one can consider equivalence relations not only on second-order objects, but on third order-objects, fourth-order objects, etc. However, Burgess uses this theory merely to motivate elements of Fine's theory, so that neither Fine nor Burgess suggest restricting attention to theories which satisfy these specific conditions.

As to its broader philosophical significance, Fine notes that one path to determinacy of truth-value would proceed through categoricity results like Theorem~\ref{thm:finethm}. For instance, discussing the specialization of this result to the ``number~of'' abstraction operator, Fine writes: ``Thus once we know the cardinality of the non-numbers, we are in a position to specify the truth of every arithmetical statement in purely logical terms'' (\cite{Fine2002} p. 86, cf. p. 93). But ultimately Fine rejected this path to determinacy of truth-value, due to concerns about one's access to claims about the cardinality of parts of the domain. He writes: ``The difficulty with this approach is to see how someone could grasp what these truth-conditions are without already having access to an infinite domain of abstract objects'' (\cite{Fine2002} p.  94). This was one reason among many that led Fine to develop an alternative approach on which one postulates ``a procedure for the construction of the domain'' (\cite{Fine2006ab} p. 90). This ``procedural postulationism'' is designed to secure not only determinacy of truth-value but also determinacy of reference (\cite{Fine2006ab} p. 89, \cite{Fine2002} p. 100). While this is no place to discuss Fine's later procedural postulationism, it's worth underscoring that the specific worry which Fine cites with respect to categoricity results does not seem so damaging to one who would insist on the requirement of natural relative categoricity. For, unlike the hypotheses of Fine's Theorem (Theorem~\ref{thm:finethm}), there are no assumptions in natural relative categoricity about the cardinality of the non-abstracts. 
However, Fine's concern might worry someone who insisted on the requirement of surjective relative categoricity as it is unclear 
how someone could have advance assurance that every object is an abstract.

\section{Hodes and Supervaluationism}\label{sec05:hodes}

The work of Hodes (\cite{Hodes1984}, \cite{Hodes1990aa}, \cite{Hodes1991}) constitutes a sustained attempt to forge techniques and ideas from Frege's \emph{Grundlagen} into a viable version of fictionalism. Hodes writes: ``[\ldots] mathematical discourse, when carried on within the mathematical object-picture, [is] a special sort of fictional discourse: numbers are fictions `created' with a special purpose, to encode numerical object-quantifiers and thereby enable us to `pull down' a fragment of third-order logic, dressing it in first-order clothing'' ( \cite{Hodes1984} p. 144).  Hodes differs from other fictionalists in his invocation of supervaluationist semantics, by which Hodes can say with fictionalists  that mathematical language doesn't refer, but also affirm that certain sentences containing mathematical vocabulary are true in a very demanding sense.

One can implement Hodes-style supervaluationism with respect to an arbitrary equivalence relation. Hodes does this only with respect to equivalence relations associated to cardinality and sets, in \cite{Hodes1990aa} and \cite{Hodes1991} respectively. But it seems to us that the definitions offered in \cite{Hodes1990aa} pp.~364-365, \cite{Hodes1991} p.~158 naturally generalize as follows:
\begin{defn}\label{defn:bivalence} 
Suppose that $E(X,Y)$ is an $L_0$-formula and $\kappa$ is a cardinal, and suppose further that $\varphi$ is an $L_0[\partial]$-sentence. Then $\varphi$ is said to be \emph{${\kappa}$-supertrue relative to $A[E]$} (resp. \emph{${\kappa}$-superfalse relative to $A[E]$}) if for all  standard models $\mathcal{M}$ of $A[E]$ of cardinality $\kappa$, it is also the case that $\mathcal{M}\models \varphi$ (resp. $\mathcal{M}\models \neg \varphi$). Further, $\varphi$ is said to be \emph{$\kappa$-bivalent relative to $A[E]$} if $\varphi$ is $\kappa$-supertrue or $\kappa$-superfalse.
\end{defn}
\noindent In this definition, we are, as in the previous sections, identifying the cardinality of the model $\mathcal{M}$ of $A[E]$ with the cardinality of its first-order part. 

As this definition makes clear, Hodes' approach was focused on the standard semantics for second-order logic. Moreover, there's a reason for this: if one instead used the Henkin semantics, then the analogue of supertruth would  reduce to provability, at least in the case where $A[E]$ has only infinite models. This is the content of the following elementary proposition, whose proof we omit since it is a simple application of the L\"owenheim-Skolem theorems.
\begin{prop}
Suppose that $E(X,Y)$ is an $L_0$-formula and  $A[E]$ is consistent but has only infinite models. Suppose further that $\varphi$ is an $L_0[\partial]$-sentence and $\kappa$ is an infinite cardinal. If all models $\mathcal{M}$ of $A[E]$ of cardinality $\kappa$ satisfy $\varphi$, then $A[E]$ proves $\varphi$.
\end{prop}
\noindent So this proposition indicates that something is gained by Hodes' invocation of standard models, namely, the idea of supertruth does not reduce to that of deduction.

Hodes' fictionalism suggests the idea of maximizing bivalence, and so Hodes himself  established various results indicating that ``[\ldots] our actual mathematical reasoning makes no use of sentences parsed by non-bivalent sentences [\ldots]'' (\cite{Hodes1990aa} p. 370, cf. Observations 4-5 pp. 367-368). With an eye towards maximizing bivalence, we  introduce the following definition of bivalence-compatibility:
\begin{defn}\label{defn:bivalenceompat}
The abstraction principle $A[E]$ is \emph{bivalence-compatible} if there is an infinite cardinal $\kappa$ such that $\varphi$ is $\kappa$-bivalent relative to $A[E]$ for each $L_0[\partial]$-sentence $\varphi$.
\end{defn}
\noindent There might be other ways of maximizing bivalence, but this at least seems like a natural enough route: the idea is that there's some size such that all standard models of that size must agree on the truth-value of sentences expressible using the abstraction operator.

It's then natural to ask which abstraction principles are bivalence compatible. By using Fine's Theorem~\ref{thm:finethm} and the Coarsening Dichotomy (Lemma~\ref{lem:dichtomylem}), one can easily show:
\begin{prop}\label{prop:iamaprop}
If~$E$ is cardinality coarsening then~$A[E]$ is bivalence compatible, and indeed $A[E]$ is $\omega_1$-bivalent. 
\end{prop}
\begin{proof}
Letting $\kappa=\omega_1$, let's note that the Coarsening Dichotomy Lemma implies that $\left|\nicefrac{P(\kappa)}{E}\right|<\kappa$. Since $\left|\nicefrac{P(\kappa)}{E}\right|<\kappa$, it  follows that for any model $\mathcal{M}$ of $A^2[E]$~(\ref{eqn:AE2}) with cardinality $\kappa$, we have that $\left|\mathrm{rng}(\partial_i)\right|<\kappa$. By Infinite Sums are Maxs~(\ref{eqn:InfSumMax}), we have that $\left|M\setminus \mathrm{rng}(\partial_1)\right|=\kappa=\left|M\setminus \mathrm{rng}(\partial_2)\right|$. Then by Fine's Theorem~\ref{thm:finethm}, we have $\mathcal{N}_1$ and $\mathcal{N}_2$ as defined in equation~(\ref{eqn:fineinduced}) are isomorphic. So for any sentence $\varphi$ of $L_0[\partial]$, we have that $\mathcal{N}_1\models \varphi$ iff $\mathcal{N}_2\models \varphi$.
\end{proof}

The situation with respect to bicardinality coarsening and cardinality coarsening on small concepts is slightly more subtle. As for bicardinality coarsening, since it was introduced in connection with surjective relative categoricity, it's natural to modify the definition of $\kappa$-bivalence (Definition~\ref{defn:bivalence}) so that attention is restricted to models of $A[E]$ where the abstraction operator is surjective. Now it's not necessarily the case that every bicardinality coarsening equivalence relation will have such a model for a given $\kappa$. For instance, the Bicardinality Principle from \S\ref{sec04} won't have such a model for $\kappa=\omega_1$, as one can easily check by reference to Proposition~\ref{prop:existencesurject}. However, for those cardinals $\kappa$ for which there are such models, Theorem~\ref{thm:coverthm} trivially implies that bicardinality coarsening suffices for $\kappa$-bivalence in the modified sense, since again isomorphism suffices for elementary equivalence.

As for cardinality coarsening on small concepts, it is simply unknown to us whether this implies bivalence compatibility. So we record the following question:
\begin{Q}\label{eqn:Q2}
Suppose that $E(X,Y)$ is an $L_0$-formula which is provably an equivalence relation in our background second-order logic and which is cardinality coarsening on small concepts. Is it necessarily the case that $A[E]$ is bivalence compatible? 
\end{Q}
\noindent If this question is answered in the negative, then it would indicate that the way in which natural relative categoricity captures the idea of determinacy of truth-value is distinct from the way in which Hodes' supervaluationism captures this idea. Of course, these two notions are extensionally distinct since an equivalence relation can be bivalence compatible simply by mimicking Hume's Principle on domains of cardinality $\omega_1$ and appealing to Proposition~\ref{prop:iamaprop}, while violating cardinality coarsening on small concepts on domains of other sizes. So the interesting direction is that which is at issue in the above question, since a negative answer would indicate that these two determinacy of truth-value ideas studied in our work and Hodes work are orthogonal to one another.

\section{Conclusions}\label{sec07}

Our goal has been to articulate various notions of relative categoricity for abstraction principles and to study which abstraction principles are accordingly relative categorical. The import of our  Theorem~\ref{thm:ncr=iia=cca} and Theorem~\ref{thm:coverthm} is that such relatively categorical abstraction principles qualitatively look like Hume's Principle. Our results contravene the general experience we have with relative categoricity. For, this notion is highly non-domain-specific in that one has relatively categorical axiomatizations of number, set, the reals, etc. But when we restrict attention down to abstraction principles, our  Theorem~\ref{thm:ncr=iia=cca} and Theorem~\ref{thm:coverthm} show that relative categoricity is tied to cardinality coarsening notions.

Finally, it's worth emphasizing the limited scope of our study. First, we have focused exclusively on abstraction principles formed from equivalence relations on unary concepts, and many natural abstraction principles like that associated to the Burali-Forti paradox concern abstraction principles on binary concepts. Antonelli (\cite{Antonelli2010ac} pp. 10-11) notes that the equivalence relation associated to the Burali-Forti paradox is permutation invariant in the sense of~(\ref{eqn:defn:piI}) once one extends this notion naturally from equivalence relations on unary concepts to binary concepts, and similarly it will be injection invariant. So this indicates that the study of relative categoricity concepts will be quite different when one goes from equivalence relations on unary concepts to those on binary concepts. For instance, this example indicates that in this more general setting we can't have that injection invariance suffices for having standard models of infinite cardinality (cf. Proposition~\ref{prop:stabileforcc}). So it's hard to predict \emph{apriori} how much of the present study holds when we pass from equivalence relations on unary concepts to those on binary concepts or $n$-ary concepts.

A similar limitation that should be emphasized is that we have focused on some specific notions of relative categoricity, namely natural relative categoricity and surjective relative categoricity. It would of course be ideal to have analogues of Theorem~\ref{thm:ncr=iia=cca} and Theorem~\ref{thm:coverthm} for the notion of relative elementary equivalence and bivalence compatibility (Definition~\ref{eqn:nateleequiv} and Definition~\ref{defn:bivalenceompat}), but we have been unable to obtain any such characterizations. Further, as indicated in Question~\ref{eqn:Q1} and Question~\ref{eqn:Q2}, even the relationships between these notions and our notions is not yet resolved. As a final note, it should be underscored that natural relative categoricity concerns a specific bijection (namely, the natural bijection) being an isomorphism. In Corollary~\ref{cor:whenarbirisnat}, we noted that this is the only isomorphism when natural relative categoricity does obtain. But it would be important to study the more general notion of relative categoricity which did not restrict attention to the natural bijection. We have been unable to establish analogues of Theorem~\ref{thm:ncr=iia=cca} and Theorem~\ref{thm:coverthm} for this more general notion. Our motivation for considering the natural bijection $\Gamma(\partial_1(X))=\partial_2(X)$ is that this seemed like a natural enough way for agents to exchange information about their abstraction operators, broadly similar in character to Parsons' agents translating their interlocutor's arithmetical vocabulary by their own arithmetical vocabulary.

\section*{Acknowledgments}

Parts of this material were presented on March 27, 2012 at the workshop on the Mathematics of Abstraction at Birkbeck, University of London; at May 23, 2014 at the workshop Abstraction: Philosophy and Mathematics at the University of Oslo; and on January 7, 2015 at the logic seminar in the department of Logic and Philosophy of Science at the University of California, Irvine. Thanks to the organizers and participants of these events and seminars, as well as to the many others who provided valuable feedback on this work. In particular, thanks to: Aldo Antonelli, Kyle Banick, Roy Cook, Samuel Eklund, Salvatore Florio, J.~Ethan Galebach, Jeremy Heis, Graham Leach-Krouse, Greg Lauro, Sarah Lawsky, {\O}ystein Linnebo, Richard Mendelsohn, Christopher Mitsch, Alexander C.~R. Oldemeier, Markus Pantsar, Jonathan Payne, Terence Parsons, Agust\'in Rayo, Sam Roberts, Marian Rogers, J.~Schatz, Gabriel Uzquiano, and Kai Wehmeier. Thanks also to the anonymous referees for many helpful comments. Walsh would also like to acknowledge the support of {\O}ystein Linnebo's European Research Council-funded project ``Plurals, Predicates, and Paradox'' and a Kurt G\"odel Research Prize Fellowship.

\bibliographystyle{alpha}
\bibliography{rel-cat.bib}

\begin{thebibliography}{Wal14b}

\bibitem[Ant10a]{Antonelli2010ac}
G.~Aldo Antonelli.
\newblock The nature and purpose of numbers.
\newblock {\em The Journal of Philosophy}, 107(4):191--212, 2010.

\bibitem[Ant10b]{Antonelli2010aa}
G.~Aldo Antonelli.
\newblock Notions of invariance for abstraction principles.
\newblock {\em Philosophia Mathematica}, 18(3):276--292, 2010.

\bibitem[Ant10c]{Antonelli2010ab}
G.~Aldo Antonelli.
\newblock Numerical abstraction via the {F}rege quantifier.
\newblock {\em Notre Dame Journal of Formal Logic}, 51(2):161--179, 2010.

\bibitem[Boo89]{Boolos1989aa}
George Boolos.
\newblock Iteration again.
\newblock {\em Philosophical Topics}, 17:5--21, 1989.
\newblock Reprinted in \cite{Boolos1998}.

\bibitem[Boo90]{Boolos1990}
George Boolos.
\newblock The standard equality of numbers.
\newblock In {\em Meaning and Method: Essays in Honor of Hilary Putnam}, pages
  261--277. Cambridge University Press, Cambridge, 1990.
\newblock Edited by George Boolos. Reprinted in \cite{Boolos1998},
  \cite{Demopoulos1995}.

\bibitem[Boo98]{Boolos1998}
George Boolos.
\newblock {\em Logic, {L}ogic, and {L}ogic}.
\newblock Harvard University Press, Cambridge, MA, 1998.

\bibitem[Bur05]{Burgess2005}
John~P. Burgess.
\newblock {\em Fixing {F}rege}.
\newblock Princeton Monographs in Philosophy. Princeton University Press,
  Princeton, 2005.

\bibitem[BW15]{Button2014ab}
Tim Button and Sean Walsh.
\newblock Ideas and results in model theory: Reference, realism, structure and
  categoricity.
\newblock arXiv:1501.00472, 2015.

\bibitem[Coo07]{Cook2007aa}
Roy~T. Cook, editor.
\newblock {\em The {A}rch\'e Papers on the Mathematics of Abstraction},
  volume~71 of {\em The Western Ontario Series in Philosophy of Science}.
\newblock Springer, Berlin, 2007.

\bibitem[Coo12]{Cook2012aa}
Roy~T. Cook.
\newblock Conservativeness, stability, and abstraction.
\newblock {\em British Journal for the Philosophy of Science}, 63:673--696,
  2012.

\bibitem[Dem95]{Demopoulos1995}
William Demopoulos, editor.
\newblock {\em Frege's {P}hilosophy of {M}athematics}.
\newblock Harvard University Press, Cambridge, 1995.

\bibitem[Ebb85]{Ebbinghaus1985aa}
Heinz-Dieter Ebbinghaus.
\newblock Extended logics: the general framework.
\newblock In Jon Barwise and Solomon Feferman, editors, {\em Model-{T}heoretic
  {L}ogics}, Perspectives in Mathematical Logic, pages 25--76. Springer, New
  York, 1985.

\bibitem[End01]{Enderton2001}
Herbert~B. Enderton.
\newblock {\em A {M}athematical {I}ntroduction to {L}ogic}.
\newblock Harcourt, Burlington, second edition, 2001.

\bibitem[Fin02]{Fine2002}
Kit Fine.
\newblock {\em The {L}imits of {A}bstraction}.
\newblock The Clarendon Press, Oxford, 2002.

\bibitem[Fin05]{Fine2005ac}
Kit Fine.
\newblock Pr\'e{}cis.
\newblock {\em Philosophical Studies}, 122(3):305--313, 2005.

\bibitem[Fin06]{Fine2006ab}
Kit Fine.
\newblock Our knowledge of mathematical objects.
\newblock In T.~Z. Gendler and J.~Hawthorne, editors, {\em Oxford Studies in
  Epistemology}, volume~1, pages 89--109. Clarendon Press, 2006.

\bibitem[Fre84]{Frege1884aa}
Gottlob Frege.
\newblock {\em Die {G}rundlagen der {A}rithmetik}.
\newblock Koebner, Breslau, 1884.

\bibitem[Fre80]{Frege1980}
Gottlob Frege.
\newblock {\em The {F}oundations of {A}rithmetic: {A} {L}ogico-{M}athematical
  {E}nquiry into the {C}oncept of {N}umber}.
\newblock Northwestern University Press, Evanston, second edition, 1980.

\bibitem[Hal87]{Hale1987aa}
Bob Hale.
\newblock {\em Abstract Objects}.
\newblock Basil Blackwell, Oxford, 1987.

\bibitem[HJ92]{Heck-Jr.1992aa}
Richard~G. Heck~Jr.
\newblock On the consistency of second-order contextual definitions.
\newblock {\em No\^us}, 26(4):491--494, 1992.

\bibitem[HJ99]{Hrbacek1999aa}
Karel Hrbacek and Thomas Jech.
\newblock {\em Introduction to {S}et {T}heory}, volume 220 of {\em Monographs
  and Textbooks in Pure and Applied Mathematics}.
\newblock Dekker, New York, third edition, 1999.

\bibitem[Hod84]{Hodes1984}
Harold Hodes.
\newblock Logicism and the ontological commitments of arithmetic.
\newblock {\em The Journal of Philosophy}, 81(3):123--149, 1984.

\bibitem[Hod90]{Hodes1990aa}
Harold Hodes.
\newblock Where do the natural numbers come from?
\newblock {\em Synthese}, 84(3):347--407, 1990.

\bibitem[Hod91]{Hodes1991}
Harold Hodes.
\newblock Where do sets come from?
\newblock {\em The Journal of Symbolic Logic}, 56(1):150--175, 1991.

\bibitem[HW00]{Hale2000}
Bob Hale and Crispin Wright.
\newblock Implicit definition and the a priori.
\newblock In Paul Boghossian and Christopher Peacocke, editors, {\em New Essays
  on the A Priori}, pages 286--319. Clarendon, 2000.
\newblock Reprinted in \cite{Hale2001}.

\bibitem[HW01]{Hale2001}
Bob Hale and Crispin Wright.
\newblock {\em The {R}eason's {P}roper {S}tudy}.
\newblock Oxford University Press, Oxford, 2001.

\bibitem[JU04]{Jane2004aa}
Ignacio Jan\'e and Gabriel Uzquiano.
\newblock Well and non-well-founded {F}regean extensions.
\newblock {\em Journal of Philosophical Logic}, 33:437--465, 2004.

\bibitem[KM85]{Keenan1985aa}
Edward~L. Keenan and Lawrence~S. Moss.
\newblock Generalized quantifiers and the expressive power of natural language.
\newblock In Johan van Benthem and Alice ter Meulen, editors, {\em Generalized
  Quantifiers in Natural Language}, pages 73--124. Floris, Dordrecht, 1985.

\bibitem[Kun80]{Kunen1980}
Kenneth Kunen.
\newblock {\em Set {T}heory}, volume 102 of {\em Studies in Logic and the
  Foundations of Mathematics}.
\newblock North-Holland, Amsterdam, 1980.

\bibitem[Kun11]{Kunen2011aa}
Kenneth Kunen.
\newblock {\em Set {T}heory}.
\newblock College Publications, London, 2011.

\bibitem[Lav99]{Lavine1999aa}
Shaughan Lavine.
\newblock Skolem was wrong.
\newblock Unpublished. Dated June, 1999.

\bibitem[Lin11]{Linnebo2011aa}
{\O}ystein Linnebo.
\newblock Chapter 6: {H}igher-order logic.
\newblock In Leon Horsten and Richard Pettigrew, editors, {\em The {C}ontinuum
  {C}ompanion to {P}hilosophical {L}ogic}, pages 105--127. Continuum, London
  and New York, 2011.

\bibitem[Mac02]{MacFarlane2002aa}
John MacFarlane.
\newblock Frege, {K}ant, and the logic in logicism.
\newblock {\em The Philosophical Review}, 111(1):25--65, 2002.

\bibitem[Man96]{Manzano1996}
Mar{\'{\i}}a Manzano.
\newblock {\em Extensions of {F}irst {O}rder {L}ogic}, volume~19 of {\em
  Cambridge Tracts in Theoretical Computer Science}.
\newblock Cambridge University Press, Cambridge, 1996.

\bibitem[Mar02]{Marker2002}
David Marker.
\newblock {\em Model {T}heory: {A}n {I}ntroduction}, volume 217 of {\em
  Graduate Texts in Mathematics}.
\newblock Springer, New York, 2002.

\bibitem[McG97]{McGee1997aa}
Vann McGee.
\newblock How we learn mathematical language.
\newblock {\em Philosophical Review}, 106(1):35--68, 1997.

\bibitem[Par90]{Parsons1990a}
Charles Parsons.
\newblock The uniqueness of the natural numbers.
\newblock {\em Iyyun}, 39(1):13--44, 1990.

\bibitem[Par08]{Parsons2008}
Charles Parsons.
\newblock {\em Mathematical Thought and Its Objects}.
\newblock Harvard University Press, Cambridge, 2008.

\bibitem[PW08]{Peters2008aa}
Stanley Peters and Dag Westerst{\aa}hl.
\newblock {\em Quantifiers in {L}anguage and {L}ogic}.
\newblock Oxford University Press, Oxford, 2008.

\bibitem[Sha91]{Shapiro1991}
Stewart Shapiro.
\newblock {\em Foundations without {F}oundationalism: {A} {C}ase for
  {S}econd-{O}rder {L}ogic}, volume~17 of {\em Oxford Logic Guides}.
\newblock The Clarendon Press, New York, 1991.

\bibitem[Sha00]{Shapiro2000ac}
Stewart Shapiro.
\newblock {\em Philosophy of Mathematics: Structure and Ontology}.
\newblock Oxford University Press, Oxford, 2000.

\bibitem[Sha05]{Shapiro2005aa}
Stewart Shapiro, editor.
\newblock {\em The Oxford Handbook of Philosophy of Mathematics and Logic}.
\newblock Oxford University Press, Oxford, 2005.

\bibitem[She91]{Sher1991aa}
Gila Sher.
\newblock {\em The {B}ounds of {L}ogic: A {G}eneralized {V}iewpoint}.
\newblock A Bradford Book. MIT Press, Cambridge, 1991.

\bibitem[Sim99]{Simpson1999}
Stephen~G. Simpson.
\newblock {\em Subsystems of {S}econd {O}rder {A}rithmetic}.
\newblock Perspectives in Mathematical Logic. Springer, Berlin, 1999.

\bibitem[SW99]{Shapiro1999ac}
Stewart Shapiro and Alan Weir.
\newblock New {V}, {ZF} and abstraction.
\newblock {\em Philosophia Mathematica}, 7(3):293--321, 1999.

\bibitem[Tar86]{Tarski1986aa}
Alfred Tarski.
\newblock What are logical notions?
\newblock {\em History and Philosophy of Logic}, 7(2):143--154, 1986.

\bibitem[VW14]{Vaananen2014aa}
Jouko V{\"a}{\"a}n{\"a}nen and Tong Wang.
\newblock Internal categoricity in arithmetic and set theory.
\newblock To appear in \emph{Notre Dame Journal of Formal Logic}, 2014.

\bibitem[Wal12]{Walsh2012aa}
Sean Walsh.
\newblock Comparing {H}ume's principle, {B}asic {L}aw {V} and {P}eano
  arithmetic.
\newblock {\em Annals of Pure and Applied Logic}, 163:1679--1709, 2012.

\bibitem[Wal14a]{Walsh2014ac}
Sean Walsh.
\newblock {F}ragments of {F}rege's \emph{{G}rundgesetze} and the constructible
  universe.
\newblock Unpublished. Dated July 11, 2014.

\bibitem[Wal14b]{Walsh2014aa}
Sean Walsh.
\newblock Logicism, interpretability, and knowledge of arithmetic.
\newblock {\em The Review of Symbolic Logic}, 7(1):84--119, 2014.

\bibitem[Wri83]{Wright1983}
Crispin Wright.
\newblock {\em Frege's {C}onception of {N}umbers as {O}bjects}, volume~2 of
  {\em Scots Philosophical Monographs}.
\newblock Aberdeen University Press, Aberdeen, 1983.

\bibitem[Wri97]{Wright1997aa}
Crispin Wright.
\newblock On the philosophical significance of {F}rege's theorem.
\newblock In Richard~G. Heck~Jr., editor, {\em Language, Thought, and Logic:
  Essays in Honour of Michael Dummett}, pages 201--244. Oxford University
  Press, Oxford, 1997.
\newblock Reprinted in \cite{Hale2001}.

\bibitem[Wri98]{Wright1998ab}
Crispin Wright.
\newblock On the harmless impredictavity of ${N}^{=}$ ({H}ume's principle).
\newblock In Matthias Schirn, editor, {\em Philosophy of {M}athematics
  {T}oday}, pages 393--368. Clarendon Press, Oxford, 1998.
\newblock Reprinted in \cite{Hale2001}.

\bibitem[Wri99]{Wright1999}
Crispin Wright.
\newblock Is {H}ume's principle analytic?
\newblock {\em Notre Dame Journal of Formal Logic}, 40(1):6--30, 1999.
\newblock Reprinted in \cite{Hale2001}.

\end{thebibliography}

\end{document}